\documentclass[11pt,reqno,twosdie]{amsart}

\usepackage{amsmath,amssymb,mathrsfs}

\usepackage[asymmetric,top=3.2cm,bottom=3.9cm,left=2.8cm,right=2.8cm]{geometry}
\usepackage{amsmath,amsfonts,amsthm,mathrsfs,amssymb,cite}
\usepackage[usenames]{color}

\usepackage{mathtools} 
\usepackage{graphics}
\usepackage{framed}

\usepackage{enumerate}

\usepackage{hyperref}
\usepackage{xcolor}
\hypersetup{
  colorlinks,
  linkcolor={blue!80!black},
  urlcolor={blue!80!black},
  citecolor={blue!80!black}
}

\usepackage{graphicx}
\usepackage{latexsym} 
\usepackage{enumerate}

\theoremstyle{plain}
\newtheorem{theorem}{Theorem}[section]
\newtheorem{lemma}[theorem]{Lemma}

\newtheorem{proposition}[theorem]{Proposition}

\theoremstyle{remark}
\newtheorem{definition}[theorem]{Definition}
\newtheorem{remark}[theorem]{Remark}
\newtheorem{example}[theorem]{Example}
\newtheorem{corollary}[theorem]{Corollary}

\renewcommand{\eqref}[1]{\textnormal{(\ref{#1})}}

\numberwithin{equation}{section}

\newcommand{\bfx}{\mathbf{x}}

\newcommand{\R}{\mathbb{R}}
\newcommand{\C}{\mathbb{C}}

\def\bsi{{\mathrm{i}}}
\def\Oh{{\mathcal  O}}

\newcommand{\mm}[1]{{\color{black} #1}}
\newcommand{\q}{\quad}   

\newcommand{\red}[1]{{\color{black} #1}}

\title[On nodal and generalized singular structures of Laplacian eigenfunctions]{On nodal and generalized singular structures of Laplacian eigenfunctions and applications to inverse scattering problems}

\author{Xinlin Cao}
\address{Department of Mathematics, Hong Kong Baptist University, Kowloon, Hong Kong, China.}
\email{xlcao.math@foxmail.com}

\author{Huaian Diao}
\address{School of Mathematics and Statistics, Northeast Normal University,
Changchun, Jilin 130024, China.}
\email{hadiao@nenu.edu.cn}

\author{Hongyu Liu}
\address{Department of Mathematics, Hong Kong Baptist University, Kowloon, Hong Kong, China.}
\email{hongyu.liuip@gmail.com, hongyuliu@hkbu.edu.hk}

\author{Jun Zou}
\address{Department of Mathematics, The Chinese University of Hong Kong, Hong Kong, China.}
\email{zou@math.cuhk.edu.hk}

\begin{document}

\begin{abstract}

In this paper, we present some novel and intriguing findings on the geometric structures of Laplacian eigenfunctions and their deep relationship to the quantitative behaviours of the eigenfunctions 
{in two dimensions}. {We introduce a new notion of generalized singular lines of the Laplacian eigenfunctions, and carefully study these singular lines and 
the nodal lines. The studies reveal that the intersecting angle between two of those lines is closely related to 
the vanishing order of the eigenfunction at the intersecting point. We establish an accurate and comprehensive quantitative characterisation of the relationship. Roughly speaking, 
the vanishing order is generically infinite if the intersecting angle is {\it irrational}, 
and the vanishing order is finite if the intersecting angle is rational. In fact, in the latter case, the vanishing order is the degree of the rationality. \mm{The theoretical findings are original and of significant interest in spectral theory. } Moreover, they are applied directly 
to some physical problems of great importance, including the inverse obstacle scattering problem and 
the inverse diffraction grating problem. It is shown in a certain polygonal setup that 
one can recover the support of the unknown scatterer as well as the surface impedance parameter by finitely many far-field patterns. Indeed, at most two far-field patterns are sufficient 
for some important applications. Unique identifiability by finitely many far-field patterns remains to be a highly challenging fundamental mathematical problem in the inverse scattering theory.}

 

\end{abstract}

\maketitle

\section{Introduction}\label{sec:Intro}

\subsection{Background}
Laplacian eigenvalue problem is arguably the simplest PDE eigenvalue problem, which is stated as finding $u\in L^2(\Omega)$ and $\lambda\in\mathbb{R}_+$ such that
\begin{equation}\label{eq:eig}
-\Delta u=\lambda u,
\end{equation}
where $\Omega$ is an open set in $\mathbb{R}^n$, $n\geq 2$, \mm{under a certain homogeneous boundary condition on $\partial\Omega$, such as the Dirichlet, Neumann or Robin condition.} 
There is a long and colourful history on the spectral theory of Laplacian eigenvalues and eigenfunctions; see e.g.  \cite{CH,weyl11,weyl12,faber,hayman,HHM,krahn,kuttler,makai,ashbaugh,Shn,Zel}. It still remains an inspiring source for many technical, practical and computational developments \cite{burdzy96,burdzy00,Rayleigh,Sapoval91,Sapoval93,Zel}. 

In this paper, we are concerned with the geometric structures of Laplacian eigenfunctions as well as their applications to inverse scattering theory. There is a rich theory on the geometric properties of Laplacian eigenfunctions in the literature; see e.g. the review papers \cite{jakobson,kuttler,grebenkov}. The celebrated Courant's nodal domain theorem states that the first Dirichlet eigenfunction doest not change sign in $\Omega$ and the $n$th eigenfunction (counting multiplicity) $u_n$ has at most $n$ nodal domains. In particular, a famous  conjecture concerning the topology of the 2nd Dirichlet eigenfunction states that in $\mathbb{R}^2$, the nodal line of $u_2$ divides $\Omega$ by intersecting its boundary at exactly two points if $\Omega$ is convex (cf. \cite{yau}). A large amount of literature has been devoted to this conjecture and significant progresses have been made in various situations \cite{payne73,lin,melas,alessandrini,hoffmann,schoen,freitas,jerison,fournais,freitas02,freitas07,kennedy,kt15}.
The ``hot spots" conjecture formulated by J. Rauch in 1974 says that the maximum of the second Neumann eigenfunction is attained at a boundary point. This conjecture was proved to be true for a class of planar domains \cite{banuelos}, \mm{but the statement may not be correct in general; see several counterexamples \cite{bass,burdzy99,burdzy05,jerison00}. The hot spots conjecture was proved recently for a new class of domains (possibly non-convex and non-Euclidean) \cite{kt19}.} Another famous longstanding problem in spectral theory is the Schiffer conjecture which states that if a Neumann eigenfunction takes constant value on the boundary, then the domain must be a ball. The Schiffer conjecture is closely related to the Pompeium property in integral geometry (cf. \cite{yau}) and has also an interesting connection to invisibility cloaking (cf. \cite{liu}). In \cite{grieser96}, the nodal set of the second Dirichlet Laplacian eigenfunction was proved to be close to a straight line when the eccentricity of a bounded and convex domain $\Omega \subset  \R^2$ is large. \mm{On the other hand, 
one may also have some estimate 
about the size of eigenfunctions, e.g., the size of the first eigenfunction can be estimated} uniformly for all convex domains; see \cite{grieser98}. \mm{Other geometrical characteristics may also be analyzed, 
e.g., the volume of  a set on which an eigenfunction is positive \cite{nadirashvili}, and lower and upper bounds for the length of the  nodal line of an eigenfunction of the Laplace operator in two-dimensional Riemannian  manifolds \cite{bruning,nadirashvili88}. }

\mm{As we see from the above, the study of the geometric structures of Laplacian eigenfunctions is intriguing and challenging.} In this paper, we present novel findings on the geometric structures of Laplacian eigenfunctions and their deep relationship to the quantitative behaviours of the eigenfunctions in $\mathbb{R}^2$. One of the motivations of our study comes from the inverse scattering theory, which is concerned with the recovery of an obstacle from the measurement of the wave pattern due to an impinging field. It is a longstanding problem in the inverse scattering theory whether one can establish the one-to-one correspondence between the geometric shape of an obstacle and its scattering wave pattern due to a single impinging wave field. This is also known as the Schiffer problem in the inverse scattering theory.

\subsection{Motivation and discussion of our main findings}

\mm{We first introduce three definitions for the descriptions of our main results.}

\begin{definition}\label{def:1}
For a Laplacian eigenfunction $u$ in \eqref{eq:eig}, a line segment $\Gamma_h\subset\Omega$ is called a \emph{nodal line} of $u$ \mm{if $u=0$ on ${\Gamma_h}$,  where $h\in\mathbb{R}_+$ signifies the length of the line segment. For a given complex-valued function $\eta\in L^\infty(\Gamma_h)$, if it holds that}
\begin{equation}\label{eq:normal1}
\partial_\nu u(\mathbf{x})+\eta(\mathbf{x})u(\mathbf{x})=0,\quad \mathbf{x}\in\Gamma_h,
\end{equation} 
then $\Gamma_h$ is referred to as a \emph{generalized singular line} of $u$. \mm{For the special case that 
$\eta\equiv 0$ in \eqref{eq:normal1}, a generalized singular line is also called a \emph{singular line} of $u$ in $\Omega$. We use ${\mathcal N}_{\Omega}^\lambda$, ${\mathcal S}_{\Omega}^\lambda$ and ${\mathcal M}_{\Omega}^\lambda$ to denote the sets of nodal, singular and generalized singular lines, respectively, 
of an eigenfunction $u$ in \eqref{eq:eig}.}
\end{definition}

According to Definition~\ref{def:1}, a singular line is obviously a generalized singular line. 
However, \mm{for unambiguity and definiteness, $\Gamma_h$ in \eqref{eq:normal1} is called 
a generalized singular line only if $\eta$ is not identically zero,} otherwise it is referred to as a singular line. 
\mm{We remark that as $u$ is (real) analytic inside $\Omega$, a nodal line or a singular line can be extended by the analytic continuation within $\Omega$. We are mainly concerned with the local properties of the eigenfunction $u$ around the intersecting point of two lines,  
and hence the length $h$ of $\Gamma_h$ does not play an essential role as long as it is positive. 
We further emphasize that no any specific boundary condition is specified for $u$ on $\partial\Omega$ in Definition~\ref{def:1},  
that is, all our subsequent results hold for a generic Laplacian eigenfunction as long as it satisfies \eqref{eq:eig}, 
therefore applicable to the particular Dirichlet, Neumann or Robin eigenfunction.}

\begin{definition}\label{def:2}
Let $\Gamma$ and $\Gamma'$ be two line segments in $\Omega$ that intersect with each other. Denote by $\theta=\angle (\Gamma, \Gamma')\in(0, 2\pi)$ the intersecting angle. Set
\begin{equation}\label{eq:normal2}
\theta=\alpha\cdot 2\pi, \ \ \alpha\in(0,1). 
\end{equation}
$\theta$ is called an \emph{irrational angle} if $\alpha$ is an irrational number; and it is called a rational angle of degree $q$ if $\alpha=p/q$ with $p, q\in\mathbb{N}$ and irreducible. 
\end{definition}

\begin{definition}\label{def:3}
Let $u$ satisfy \eqref{eq:eig} and be a nontrivial eigenfunction. For a given point $\mathbf{x}_0 \in \Omega$, if there  exits a number $N \in {\mathbb N}\cup\{0\}$ such that  
\begin{equation}\label{eq:normal3}
	\lim_{r\rightarrow +0} \frac{1}{r^m} \int_{B(\mathbf{x}_0,r)}\, |u(\mathbf{x})|\, {\rm d} \mathbf{x}=0\ \ \mbox{for}\ \ m=0,1,\ldots, N,
\end{equation}
	where $B(\mathbf{x}_0,r)$ is a disk  centered at $\mathbf{x}_0$ with radius $r\in\mathbb{R}_+$, we say that $u$ vanishes at $\mathbf{x}_0$ up to the order $N$. The largest possible $N$ such that \eqref{eq:normal3} is fulfilled is called the vanishing order of $u$ at $\mathbf{x}_0$, and we write 
\begin{equation}\label{eq:normal4}
\mathrm{Vani}(u; \mathbf{x}_0)=N. 
\end{equation}
If \eqref{eq:normal3} holds for any $N\in\mathbb{N}$, then we say that the vanishing order is infinity. 
\end{definition}


{Since $u$ to \eqref{eq:eig} is analytic in $\Omega$, it is straightforward to verify that $\mathrm{Vani}(u; \mathbf{x}_0)$ is actually the lowest order of the  homogeneous polynomial in the Taylor series expansion of $u$ at $\mathbf{x}_0$.} Moreover, by the strong unique continuation principle, we know that if the vanishing order of $u$ is infinity at a given point $\mathbf{x}_0\in\Omega$, then $u$ is identically zero in $\Omega$. 


To motivate our study, we next consider a simple example which connects the vanishing order of an eigenfunction with the intersecting angle of its nodal lines. Set 
\begin{equation}\label{eq:ab1}
u(\mathbf{x})=J_n(\sqrt{\lambda} \,r)\sin n\theta,\quad \mathbf{x}=(x_1,x_2)=r(\cos\theta,\sin\theta)\in\Omega,
\end{equation}
where $J_n$ is the first-kind Bessel function of order $n\in\mathbb{N}$ (cf. Section 3.4 of \cite{CK}). $u(\mathbf{x})$ is a single spherical wave mode and satisfies \eqref{eq:eig}, and we can verify that 
\begin{equation}\label{eq:V1}
\mathrm{Vani}(u; \mathbf{0})=n. 
\end{equation}
In particular, it is noted that if one considers $u$ in a central disk $B_{r_0}$ with $\sqrt{\lambda}\,r_0$ being a root of $J_n(t)$ or $J_n'(t)$, then $u$ is actually a Dirichlet or Neumann eigenfunction in $\Omega=B_{r_0}$. However, 
we are more interested in the nodal lines of $u$, and it can be easily seen that 
\begin{equation}\label{eq:nodal1}
\Gamma_h^m:=\{\mathbf{x}=r(\cos\theta_m, \sin\theta_m); -h<r<h, \theta_m=\frac{m}{n}\pi\},\ \ m=0, 1, 2,\ldots, 2n-1. 
\end{equation}
The nodal lines in \eqref{eq:nodal1} all intersect at the origin and the intersecting angle between any two of them is rational of degree $n$. Clearly, this simple example reveals an intriguing
connection between the intersecting angle of two nodal lines and the vanishing order of the eigenfunction at the intersecting point. The aim of the present paper is to establish an accurate 
and comprehensive characterisation of such a relationship in the most general scenario. Roughly speaking, 
\mm{we shall show that the vanishing order is generically infinity if the intersecting angle is {\it irrational}, 
and the vanishing order is finite if the intersecting angle is {\it rational}.} In the latter case, the vanishing order is actually the degree of rationality of the intersecting angle.
 The result is not only established for the nodal lines, but also 
for the generalized singular lines. \mm{Hence our study uncovers} a deep relationship between the nodal and generalized singular structures of the Laplacian eigenfunctions and the quantitative behaviours of the eigenfunctions. To the best 
of our knowledge, it is the first time in the literature to present a systematic study 
of such intriguing connections between the vanishing orders 
of Laplacian eigenfunctions and the intersecting angles of their nodal/generalized singular lines. 
Hence, these results should be truly original and of significant interest in the spectral theory of Laplacian eigenfunctions, 
and possibly very closely related Helmholtz and Maxwell eigenfunctions as well.
In order to establish the aforementioned results, we make essential use of the spherical wave expansion of the eigenfunction, which in combination with the homogeneous conditions on the intersecting lines can yield certain recursive formulae of the Fourier coefficients. In order to trigger off the recursion for us to achieve the desired vanishing order of the eigenfunction, we need to show the vanishing of the first few polynomial terms (basically up to the third order) of the eigenfunction. For this part, we develop a ``localized" argument, which makes use of tools from microlocal analysis to derive accurate characterisation of the singularities of the eigenfunction at the intersecting point in the phase space.
\mm{This involves rather delicate and technical analysis, but it only requires the ``local" information of the eigenfunction in a corner region formed by the intersecting lines. This is in sharp contrast to the Fourier expansion, which requires the ``global" information of the eigenfunction around the intersecting point. In principle, the arguments that are developed in this work can be used to extend our study to the case with general second order elliptic operators as well as to the case that the nodal or singular lines are lying on the boundary $\partial\Omega$ of the domain.} However, we choose in this work 
to stick to the fundamental case with the Laplacian eigenfunctions and the nodal or generalised singular lines lying within the domain $\Omega$, and study the aforementioned technical extensions in our future work. 

\mm{In addition to their theoretical beauty and profundity,
our new spectral findings in this work can be directly applied 
to some physical problems of great practical importance, including the inverse obstacle scattering problem and 
the inverse diffraction grating problem. } 
By using the new critical connection 
between the intersecting angles of the nodal/generalized singular lines and the vanishing order of the eigenfunctions, we establish in a certain polygonal setup that one can recover the support of the unknown scatterer as well as the surface impedance parameter 
by finitely many far-field patterns. In fact, two far-field patterns are sufficient 
for some important applications. It is well known that unique identifiability by finitely many far-field patterns remains a highly challenging fundamental mathematical problem in the inverse scattering theory. Using the new spectral findings, 
we are able not only to establish the unique identifiability results for some open inverse scattering problems, 
especially for the impedance case, but also to develop a completely new approach 
that can treat the unique identifiability issue for several inverse scattering problems in a unified manner, especially 
in terms of general material properties. Most existing analytical theories for the unique identifiability 
of inverse scattering problems need to handle each special material property very differently.
We shall give more background introduction in Section~\ref{sec8} about these practical problems so that we can first focus on the theoretical study of the nodal and singular structures of the Laplacian eigenfunctions. 

{The rest of the paper is organized as follows. In Sections \ref{sec:irra} and \ref{sec:irrational gene}, we consider the case that the intersecting angle is irrational and show that the vanishing order is infinity. In Sections \ref{sec:3}, \ref{sec:pro1} and \ref{sec:pro2}, we study the case that the intersecting angle is rational and the vanishing order is finite. Section \ref{sec:3} is devoted to the presentation and discussions of the main results, whereas Sections \ref{sec:pro1} and \ref{sec:pro2} are concentrated on the corresponding rigorous proofs. 
Section~\ref{sect:discussion} discusses a generic condition required in Sections \ref{sec:irra}--\ref{sec:pro2}. 
 In Section \ref{sec8}, we establish the unique recovery results for the inverse obstacle problem and the inverse diffraction grating problem by at most two incident waves. }

\section{Irrational intersection and infinite vanishing order: two intersecting nodal and singular lines}\label{sec:irra}

In this section, we consider a relatively simple case that two nodal lines or two singular lines intersect at an irrational angle. We show that in such a case, the vanishing order of the eigenfunction at the intersecting point is generically infinite, and hence it is identically vanishing in $\Omega$. 

\begin{theorem}\label{Th:u}
Let $u$ be a Laplacian eigenfunction to \eqref{eq:eig}. If there are two nodal lines $\Gamma_h^+$ and $\Gamma_h^-$ from ${\mathcal N}^\lambda_{\Omega }$ such that 
\begin{equation}\label{eq:cond1}
\Gamma_h^+\cap\Gamma_h^-=\mathbf{x}_0\in\Omega\quad\mbox{and}\quad \angle (\Gamma_h^+,\Gamma_h^-)=\alpha\cdot 2\pi, 
\end{equation}
where $\alpha\in (0, 1)$ is irrational. \mm{Then the vanishing order of $u$ at $\mathbf{x}_0$ is infinite, namely} 
\begin{equation}\label{eq:cond2}
\mathrm{Vani}(u; \mathbf{x}_0)=+\infty. 
\end{equation}
\end{theorem}

In order to prove Theorem~\ref{Th:u}, \mm{we need some auxiliary results about reflection principles of 
nodal and singular lines from the following two lemmas.} In what follows, for a line segment $\Gamma\subset\mathbb{R}^2$, we define $\mathcal{R}_\Gamma$ to be the (mirror) reflection in $\mathbb{R}^2$ with respect to the line containing $\Gamma$.

\begin{lemma}\label{lem:reflection}
Let $u$ be a Laplacian eigenfunction to \eqref{eq:eig}. There hold the following reflection principles:
\begin{enumerate}
\item Let $\Gamma\in\mathcal{N}_\Omega^\lambda$ (resp. $\Gamma\in\mathcal{S}_\Omega^\lambda$) and $\Gamma'\in \mathcal{N}_\Omega^\lambda\cup\mathcal{S}_\Omega^\lambda$. If $\widetilde{\Gamma}=\mathcal{R}_{\Gamma'}(\Gamma)\subset\Omega$, then $\widetilde\Gamma\in\mathcal{N}_\Omega^\lambda$ (resp. $\widetilde\Gamma\in\mathcal{S}_\Omega^\lambda$);

\item Let $\Gamma\in\mathcal{M}_\Omega^\lambda$ with $\partial_\nu u+\eta u=0$ on $\Gamma$ and $\Gamma'\in\mathcal{N}_\Omega^\lambda$. If $\widetilde{\Gamma}=\mathcal{R}_{\Gamma'}(\Gamma)\subset\Omega$, then $\widetilde\Gamma\in\mathcal{M}_\Omega^\lambda$ satisfies $\partial_{\widetilde{\nu}} u+\widetilde{\eta}u=0$ on $\widetilde{\Gamma}$, where $\widetilde{\nu}=\mathcal{R}_{\Gamma'}(\nu)$ and $\widetilde{\eta}=\mathcal{R}_{\Gamma'}(\eta)$. 
\end{enumerate}
\end{lemma}

\mm{The reflection principles are rather standard for the Laplacian \cite{Liu-Zou,Liu-Zou3}.} 
The first reflection principle in Lemma~\ref{lem:reflection} shall be used in the proof of Theorem~\ref{Th:u}, 
whereas the second one is needed in our subsequent study. 

\begin{lemma}\label{lem:alpha}
	Let $0<\alpha_1<1$ be an irrational number. Define
	\begin{equation}\label{eq:21lim}
		\alpha_{n+1}=1-\left\lfloor  \frac{1}{\alpha_n} \right \rfloor   \alpha_{n},\quad n=1,2,\ldots, 
	\end{equation}		
	where $\lfloor \cdot  \rfloor$ is the floor function. Then it holds that 
	$$
	\lim_{n\rightarrow \infty} \alpha_n=0. 
	$$
\end{lemma}

\begin{proof}
We prove this lemma by  contradiction. Since 
$$
0<\frac{1}{\alpha_n} -\left\lfloor  \frac{1}{\alpha_n} \right \rfloor  <1 ,
$$
we know that the sequence $\{\alpha_n \}$ is bounded below and decreasing. Suppose that 
$$
\lim_{n\rightarrow \infty} \alpha_n=\beta_0>0. 
$$
 It is easy to see that $\{\alpha_n \}\subset  \R \backslash {\mathbb Q}$, where  ${\mathbb Q}$ is the set of rational numbers. Since $1/\alpha_n$ is not an integer, we know from \cite[p. 15, Eq.(2.1.7)]{Titchmarsh} that the Fourier series expansion of $\lfloor 1/\alpha_n  \rfloor$ as follows
\begin{equation}\label{eq:ll1}
\left\lfloor  \frac{1}{\alpha_n} \right \rfloor=\frac{1}{\alpha_n}-\frac{1}{2}+\frac{1}{\pi }\sum_{k=1}^\infty \frac{\sin(2k\pi /\alpha_n)}{k}. 
\end{equation}
Taking limits as $n\rightarrow \infty$ on the both sides of \eqref{eq:ll1}, we can prove that
$$
\lim_{n\rightarrow \infty } \left\lfloor  \frac{1}{\alpha_n} \right \rfloor= \left\lfloor  \frac{1}{\beta_0} \right \rfloor.
$$
Moreover, by taking the limits on both sides of \eqref{eq:21lim}, we also have that
\begin{equation}\label{eq:ll2}
\beta_0=1-\left\lfloor  \frac{1}{\beta_0} \right \rfloor \beta_0. 
\end{equation}
Dividing $\beta_0$ on the both sides of \eqref{eq:ll2}, we finally arrive at a contradiction
$$
1=  \frac{1}{\beta_0} -\left\lfloor  \frac{1}{\beta_0} \right \rfloor,
$$
which completes the proof. 
\end{proof}

We are now ready to present the proof of Theorem~\ref{Th:u}. 

\begin{proof}[Proof of Theorem~\ref{Th:u}]

By a rigid motion if necessary, we can assume without loss of generality that $\mathbf{x}_0=\mathbf{0}$ is the origin and $\Gamma_h^-$ lies in the $x_1^+$-axis while $\Gamma_h^{+}$ has the angle $2\alpha \pi $ away from $\Gamma^-_h$ in the  clockwise direction. Since we are mainly concerned with the local properties, it is assumed that $h\in\mathbb{R}_+$ is sufficiently small such that $\mathcal{R}_{\Gamma_h^+}(\Gamma_h^-)\Subset\Omega$. By Lemma~\ref{lem:reflection}, we see that
\[
\Gamma_{1,h}:=\mathcal{R}_{\Gamma_h^+}(\Gamma_h^-)\in\mathcal{N}_\Omega^\lambda. 
\]
By repeating the above reflection, one can find a nodal line $\Gamma_{{\ell_1},h}$ which is the nearest one to $x_1^+$-axis in the sense that
\[
{\Gamma_{\ell_1,h}} \in\mathcal{N}_\Omega^\lambda,\quad \ell_1=\left\lfloor  \frac{1}{\alpha} \right \rfloor \quad\mbox{and}\quad  \alpha_{1}=1-\left\lfloor  \frac{1}{\alpha} \right \rfloor   \alpha < \alpha, \ \ \angle(\Gamma_{\ell_1, h},\Gamma_h^-)=\alpha_1\cdot 2\pi.
\]
Next, by replacing $\Gamma_h^+$ with $\Gamma_{\ell_1,h}$ and repeating the above reflection argument, one can find a nodal line $\Gamma_{{\ell_2},h}$ which is the nearest one to $x_1^+$-axis such that
\[
{\Gamma_{\ell_2,h}} \in\mathcal{N}_\Omega^\lambda,\quad \ell_2=\left\lfloor  \frac{1}{\alpha_1} \right \rfloor \quad\mbox{and}\quad  \alpha_{2}=1-\left\lfloor  \frac{1}{\alpha_1} \right \rfloor   \alpha_1 < \alpha_1, \ \ \angle(\Gamma_{\ell_2, h},\Gamma_h^+)=\alpha_2\cdot 2\pi.
\]
Clearly, by repeating this reflection argument, one can find a series of nodal lines $\Gamma_{\ell_n, h}$ such that
\begin{equation}\label{eq:aa0}
\ell_n=\left\lfloor  \frac{1}{\alpha_{n-1} } \right \rfloor\ \ \mbox{and}\ \ \alpha_n=1-\left\lfloor  \frac{1}{\alpha_{n-1} } \right \rfloor   \alpha_{n-1}, \ \ \angle(\Gamma_{\ell_n, h},\Gamma_h^+)=\alpha_n\cdot 2\pi.
\end{equation}
Then by Lemma \ref{lem:alpha} we have 
\begin{equation}\label{eq:aa1}
\lim_{n\rightarrow \infty} \alpha_n=0. 
\end{equation}
By combining \eqref{eq:aa0}, \eqref{eq:aa1} and the reflection principle in Lemma~\ref{lem:reflection}, one can actually show that there exists a dense set of nodal lines around the origin. Hence, by the continuity of $u$ one readily has that $u$ is identically zero. 
This completes the proof of Theorem~\ref{Th:u}.
\end{proof}

The next theorem is concerned with the intersection of two singular lines.

\begin{theorem}\label{th:28}
Let $u$ be a Laplacian eigenfunction to \eqref{eq:eig}. If there are two singular  lines $\Gamma_h^+$ and $\Gamma_h^-$ from ${\mathcal S}^\lambda_{\Omega }$ such that 
\begin{equation}\label{eq:cond1}
\Gamma_h^+\cap\Gamma_h^-=\mathbf{x}_0\in\Omega\quad\mbox{and}\quad \angle (\Gamma_h^+,\Gamma_h^-)=\alpha\cdot 2\pi, 
\end{equation}
where $\alpha\in (0, 1)$ is irrational, then there hold that 
		\begin{align*}
		&\mathrm{Vani}(u; \mathbf{x}_0)= 0,\quad\ \ \ \, \mbox{if } u({\mathbf x}_0) \neq 0;  \\
		&\mathrm{Vani}(u; \mathbf{x}_0)= +\infty,\quad \mbox{if } u({\mathbf x}_0) = 0.
		\end{align*}
		Moreover, if $u({\mathbf x}_0) \neq 0$, we have the following expansion of $u$ in a neighborhood of  $\mathbf x_0$:
		\begin{equation*}
	u(\mathbf{x})=u({\mathbf x}_0) J_0(\sqrt{\lambda}r),\quad \mathbf  x=\mathbf x_0+ r (\cos \theta, \sin \theta). 
	\end{equation*}
	where $J_0(t)$ is the zeroth  Bessel function of the first kind. 
\end{theorem}

\mm{In order to prove Theorem \ref{th:28}, we need some auxiliary results from the following three lemmas, 
especially about the spherical wave expansion, for which we refer to \cite{CK} for more details.
In what follows, $\mathrm{i}:=\sqrt{-1}$ is used for the imaginary unit. 
}

\begin{lemma}\cite[Section 3.4]{CK}\label{lem:25}
	Suppose that $u$ is an eigenfunction to  \eqref{eq:eig}, then $u$ has the following spherical wave expansion in polar coordinates around the origin:
	\begin{equation}\label{eq:uSex}
	u(\mathbf{x})=\sum_{n=0}^\infty\left( a_n e^{\bsi n \theta }+b_n e^{-\bsi n \theta }\right)J_n\left(\sqrt{\lambda}  r\right ),
	\end{equation}	
	where $\mathbf{x}=(x_1,x_2)=r(\cos \theta, \sin\theta)\in \R^2$, $\lambda$ is the corresponding eigenvalue of \eqref{eq:eig}, $a_n$ and $b_n$ are constants, and $J_n(t)$ is the $n$-th  Bessel function of the first kind. \end{lemma}

\begin{lemma}\label{lem:26}
	Let $\Gamma$ be a line segment that can be parameterized in polar coordinates as 
	$\mathbf{x} \in \Gamma$, where $\mathbf{x}=r(\cos \theta, \sin \theta)$ 
	with $0\leq r < \infty \mbox{ and } \theta ~\mbox{fixed}$.  Let $\nu$ be the unit normal vector to $\Gamma$,  then it holds that 
	$$
	\frac{\partial u}{\partial \nu }= \pm \frac{1}{r} \frac{\partial u}{\partial \theta }. 
	$$ 
\end{lemma}

\begin{proof}
	 Using the polar coordinates and {\color{black} the chain rule}, we have
	\begin{align*}
	\frac{\partial u }{ \partial x_1}&=  \frac{\partial u }{ \partial r} \cos \theta -\frac{\sin  \theta   }{r} \frac{\partial u }{ \partial \theta},\quad 
	\frac{\partial u }{ \partial x_2}=  \frac{\partial u }{ \partial r} \sin \theta+\frac{\cos \theta }{r} \frac{\partial u }{ \partial \theta}. 
	\end{align*}
	Thus 
	\begin{align} \label{eq:directional derivatitve}
	\begin{split}
	\frac{\partial u }{ \partial \nu}&=\left( \frac{\partial u }{ \partial r}\Big|_{\Gamma } \cos \theta -\frac{\sin \theta   }{r} \frac{\partial u }{ \partial \theta}\Big |_{\Gamma} \right )\cos \varphi 
	+\left( \frac{\partial u }{ \partial r}\Big |_{\Gamma }  \sin \theta +\frac{\cos \theta }{r} \frac{\partial u }{ \partial \theta}\Big |_{\Gamma }  \right)\sin  \varphi, \\
		\end{split}
	\end{align}
	where $\nu =(  \cos \varphi, \sin \varphi )$ denotes the  unit normal vector to $\Gamma$ .  Using the fact $|\varphi-
	\theta|=\pi/2$, we complete the proof.
\end{proof}

\begin{lemma}\label{lem:27}
	Suppose that for $0<h \ll 1$ and $t\in (0,h)$, 
	\begin{equation}\label{eq:23pro}
	\sum_{n=0}^\infty \alpha_n J_n(t)=0,
	\end{equation}
	where $J_n(t)$ is the $n$-th  Bessel function of the first kind.  Then
	$$
	\alpha_n=0, \quad n=0,1,2,\ldots. 
	$$
\end{lemma}

\begin{proof}
    From \cite{CK}, we know that 
	\begin{equation}\label{eq:jnt}
	J_n(t)=\frac{t^n}{2^n n!} \left( 1+ \sum_{p=1}^\infty \frac{(-1)^p n!}{p!(n+p)!} \left(\frac{t }{2} \right)^{2p}  \right  ) . 
	\end{equation}
	Substituting \eqref{eq:jnt}  into  \eqref{eq:23pro} and comparing the coefficient of $t^n$ ($n=0,1,2,
	\ldots$), we can prove this lemma. 
\end{proof}


Now we are in a position to prove Theorem \ref{th:28}.
\begin{proof}[Proof of Theorem~\ref{th:28}] {\color{black} Without loss of generality, we assume that two  singular  lines $\Gamma_h^+$ and $\Gamma_h^-$ intersect  with each other at the origin. Using the reflection principle and a similar argument to the proof of Theorem \ref{Th:u}, for any line segment $\Gamma \Subset \Omega=\{\mathbf{x}; \mathbf{x}=r(\cos \beta, \sin \beta),\, 0\leq r\leq h \}$ pointed out from the origin we can show that
	$$
\frac{\partial u} {\partial \nu_\Gamma } \equiv 0 \mbox{ in } \Omega, 	$$
	where $\nu_{\Gamma}$ is a unit normal vector to $\Gamma$. 
	\mm{Recalling the expansion \eqref{eq:uSex}, it is easy to see from Lemma \ref{lem:26} that}
	\begin{equation}\label{eq:23utheta}
	\frac{\partial u}{\partial \theta}\Big |_{\Gamma} =\sum_{n=0}^\infty \bsi n \left( a_n e^{\bsi n \beta  }- b_n e^{-\bsi n \beta  }\right)J_n\left(\sqrt{\lambda}  r\right )=0.
	\end{equation}
	Taking $\beta=0$ in \eqref{eq:23utheta},  we derive from Lemma \ref{lem:27} that 
	$$
	\sum_{n=0}^\infty \bsi n \left( a_n - b_n \right)J_n\left(\sqrt{\lambda}  r\right )=0,
	$$
thus 
	$$
	\bsi n \left( a_n - b_n \right)=0, \quad n=1,2,\ldots. 
	$$
	Moreover, evaluating \eqref{eq:23utheta}  at  $ \beta=\alpha \pi $ where $\alpha \in \R \backslash \mathbb Q$,} we can deduce that
	$$
	\bsi n \left( a_n e^{\bsi n \alpha \pi }- b_n e^{-\bsi n \alpha \pi  } \right)=0, \quad n=1,2,\ldots. 
	$$
	Hence $a_n$ and $b_n$ satisfy 
	$$
	\begin{bmatrix}
	1&-1\\
	e^{\bsi n \alpha \pi }&-e^{-\bsi n \alpha \pi }\end{bmatrix} \begin{bmatrix}
	a_n \\ b_n
	\end{bmatrix}=0,\quad n=1,2,\ldots, 
	$$
	which readily implies that $a_n=b_n=0 $ for $n=1,2,\cdots$, {in view of Lemma \ref{lem:27}. }
	Therefore, $u(\mathbf{x})$ has the simplified form:
	\begin{equation*}
	u(\mathbf{x})=(a_0+b_0)J_0(\sqrt{\lambda}r).
	\end{equation*}
	Finally, from the assumptions in the theorem we can easily see 
		\begin{align*}
		&a_0+b_0=  u({\mathbf 0} ) \neq 0,\quad \mbox{if } u({\mathbf 0} ) \neq 0;  \\
		&a_0+b_0=  u({\mathbf 0} ) = 0,\quad \mbox{if } u({\mathbf 0}) = 0,
		\end{align*}
		 which complete the proof.
	\end{proof}
	
	

\section{Rational intersection and finite vanishing order}\label{sec:3}

In this section, we consider the general case that two line segments from Definition~\ref{def:1} which intersect at a rational angle. Throughout the present section, we let $\Gamma_h^+$ and $\Gamma_h^-$ signify the two line segments which could be either one of the three types: nodal line, singular line or generalized singular line. It is also assumed that for a generalized singular line of the form \eqref{eq:normal1}, the parameter $\eta$ is a constant. Nevertheless, we would like to point out that for the case that $\eta$ is a function in the generalized singular line, we can derive similar conclusions, but through more tedious and subtle calculations. We shall address 
this point more in Section~\ref{sec:pro1}. Let $\eta_1$ and $\eta_2$ signify the parameters associated with $\Gamma_h^-$ and $\Gamma_h^+$ respectively if they are generalized singular lines. Set 
\begin{equation}\label{eq:angle1}
\angle(\Gamma_h^+,\Gamma_h^{-})=\alpha\cdot \pi, \quad \alpha\in (0, 2),
\end{equation}
where $\alpha$ is a rational number of the form $\alpha=p/q$ with $p, q\in\mathbb{N}$ and irreducible. Since the Laplace operator $-\Delta$ is invariant under rigid motions, without loss of generality, we assume throughout the rest of this work that
\begin{equation}\label{eq:angle}
	\Gamma_h^+\cap\Gamma_h^-=\mathbf{0}\in\Omega,
	\end{equation}
and $\Gamma^-_h$ coincides with the  ${x_1}^+$-axis while $\Gamma^+_h$ has the angle $\alpha\cdot \pi$ away from $\Gamma^-_h$ in the anti-clockwise direction; see Figure \ref{fig1} for a schematic illustration. 
 \begin{figure}
	\centering
	\includegraphics[width=0.3\textwidth]{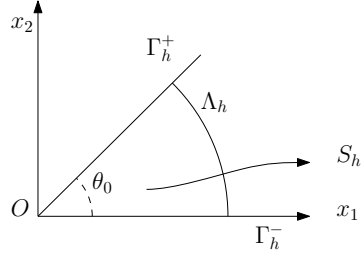}
	\caption{Schematic of the geometry of two intersecting lines with an angle $\alpha \pi$, where $\alpha\in (0, 1)\cap \mathbb{Q}$.}
	\label{fig1}
\end{figure}    
Finally, we mainly deal with the case that the two intersecting line segments $\Gamma_h^+$ and $\Gamma_h^-$ form an angle satisfying
\begin{equation}\label{eq:angle2a}
\angle(\Gamma_h^+,\Gamma_h^{-})=\alpha\cdot \pi, \quad \alpha\in (0, 1),
\end{equation}
and {the other case with $1<\alpha<2$ can be reduced to the previous case by a straightforward argument. Indeed, for $1<\alpha<2$, we know that $\Gamma_h^+$ belongs to the half-plane of $x_2<0$ (see Fig.~\ref{fig1}). Let $\widetilde{\Gamma}_h^+$ be the extended line segment of length $h$ in the half-plane of $x_2>0$. By the analytic continuation, we know that $\widetilde{\Gamma}_h^+$ is of the same type of $\Gamma_h^+$. Hence, instead of studying the intersection of $\Gamma_h^+$ and $\Gamma_h^-$, one can study the intersection of $\widetilde{\Gamma}_h^+$ and $\Gamma_h^-$, and its relations to the vanishing order of the eigenfunction. Clearly, the angle between $\widetilde{\Gamma}_h^+$ and $\Gamma_h^-$ satisfies \eqref{eq:angle2a}. 
 }

For a clear exposition, the rest of the section is devoted to the presentation and discussion of our main results, and their proofs shall be postponed to Sections~\ref{sec:pro1} and \ref{sec:pro2}. In Section~\ref{sec:pro1}, we consider the case where the vanishing of the eigenfunction is up to the third order, whereas in Section~\ref{sec:pro2}, we consider the case of general vanishing orders. 

 \begin{theorem}\label{th:41}
Let $u$ be a Laplacian eigenfunction to \eqref{eq:eig}. Suppose that there are two generalized singular lines $\Gamma_h^+$ and $\Gamma_h^-$ from ${\mathcal M}^\lambda_{\Omega }$ such that \eqref{eq:angle} and \eqref{eq:angle2a} hold. 
	Assume that $\eta_1 \equiv C_1$ and $\eta_2 \equiv  C_2$, where $C_1 $ and $C_2$ are two constants.  Then  the  Laplacian eigenfunction $u$ vanishes up to the order $N$ at $\mathbf{0}$ with respect to $\angle (\Gamma_h^+,\Gamma_h^-)=\alpha\cdot \pi$: 
	\begin{equation}\label{eq:th1}
			N= n,\quad \mbox{if } u(\mathbf 0)  =0 \mbox{ and } \alpha  \neq \frac{q  }{p}, \, p=1,\ldots, n-1,  
	\end{equation}
	where $n\in\mathbb{N}$, $n\geq 3$ and for a fixed $p$, $ q=1,2,\ldots, p-1.$
	
	\end{theorem}

	In Theorem \ref{th:41}, we require that $n\geq 3$. That means, we exclude the special case that the intersecting angle is $\pi/2$. Nevertheless, we shall discuss this special case in Remark \ref{re:38} with more details in what follows. In the next two theorems, we consider the case of two  intersecting singular and nodal lines respectively.

	\begin{theorem}\label{th:47}
	Let $u$ be a Laplacian eigenfunction to \eqref{eq:eig}. Suppose that there are two nodal lines $\Gamma_h^+$ and $\Gamma_h^-$ from ${\mathcal N}^\lambda_{\Omega }$ such that  \eqref{eq:angle} and \eqref{eq:angle2a} hold. 
	Then  the  Laplacian eigenfunction $u$ vanishes up to the order $N$ at $\mathbf{0}$ with respect to $\angle (\Gamma_h^+,\Gamma_h^-)=\alpha\cdot \pi$:
	\begin{subequations}
		\begin{align*}
			&N= n,\quad \mbox{if } \alpha \neq \frac{q  }{p}, \, p=1,\ldots, n-1,  
				\end{align*}
				where $n\in\mathbb{N}$, $n\geq 3$ and for a fixed $p$, $ q=1,2,\ldots, p-1.$
	\end{subequations}
	\end{theorem}

		\begin{theorem}\label{th:48}
	Let $u$ be a Laplacian eigenfunction to \eqref{eq:eig}. Suppose that there are two singular lines $\Gamma_h^+$ and $\Gamma_h^-$ from ${\mathcal S}^\lambda_{\Omega }$ such that  \eqref{eq:angle} and \eqref{eq:angle2a} hold. 
	Then  the  Laplacian eigenfunction $u$ vanishes up to the order $N$ at $\mathbf{0}$ with respect to $\angle (\Gamma_h^+,\Gamma_h^-)=\alpha\cdot \pi$:
	\begin{subequations}
		\begin{align*}
			&N= n,\quad \mbox{if }  u(\mathbf 0)  =0 \mbox{ and } \alpha \neq \frac{q  }{p}, \, p=1,\ldots, n-1,  
				\end{align*}
				where $n\in\mathbb{N}$, $n\geq 3$ and for a fixed $p$, $ q=1,2,\ldots, p-1.$
	\end{subequations}
	\end{theorem}

\begin{example} Let  $\Omega:=\{ (x_1,x_2) \in \R^2~|~ -2\pi \leq x_1 \leq 2\pi, \, -4\pi \leq x_2 \leq 4\pi \}$ be a rectangle.  It is easy to see that 
	$$
	u(x_1,x_2)=\sin x_1 \sin 2 x_2
		$$
		is an eigenfunction to \eqref{eq:eig} with a homogeneous Dirichlet boundary condition on $\partial \Omega$. The corresponding eigenvalue is $\lambda=5$. One pair of nodal lines of $u$ in $\Omega$ are $\{(x_1,x_2)~|~ x_2=0, -2\pi+h\leq x_1\leq 2\pi -h\}$   and $\{(x_1,x_2)~|~ x_1=0, -4\pi+h\leq x_2\leq 4\pi -h\}$ for a fixed $0<h<2\pi $, which are perpendicular to each other at the origin. Therefore from Theorem \ref{th:47}, since $\angle (\Gamma_h^+,\Gamma_h^-)=\pi/2$ which implies that $\alpha \neq 1$,  the vanishing order $N$ at the origin satisfies $N=2$.  In fact, by the explicit expression of $u$, we know that the order of lowest nontrivial homogeneous polynomial of the Taylor expansion of $u$ at the origin is 2, which coincides with the conclusion given by Theorem \ref{th:48}. 
		
\end{example}

We now proceed to consider that a nodal line intersects with a generalized singular line. {Without loss of generality, we can assume that $\Gamma_h^-$ is the generalized singular line, while $\Gamma_h^+$ is the nodal line. }

\begin{theorem}\label{Th:49}
	Let $u$ be a Laplacian eigenfunction to \eqref{eq:eig}. Suppose that a generalized singular  line $\Gamma_h^+ \in {\mathcal M}^\lambda_{\Omega }$ intersects with a nodal  line $\Gamma_h^- \in {\mathcal N}^\lambda_{\Omega }$ at $\mathbf 0$ with the angle $\angle (\Gamma_h^+,\Gamma_h^-)=\alpha\cdot \pi$. Assume that the boundary parameter $\eta_2 \equiv C_2 $ on $\Gamma_h^+$ is a constant.  Then  the  Laplacian eigenfunction $u$ vanishes up to the order $N$ at $\mathbf{0}$ with respect to $\angle (\Gamma_h^+,\Gamma_h^-)=\alpha\cdot \pi$:
		\begin{align*}
			&N=n,\quad \mbox{if }   \alpha \neq \frac{2q +1 }{2p}  , \, p=1,\ldots, n-1,  
				\end{align*}
				where $n\in\mathbb{N}$, $n\geq 2$  and for a fixed $p$, $ q=0,1,\ldots, p-1.$
		\end{theorem} 
	

Next, we consider the intersection of a singular line and a generalized singular line. Similar to Theorem \ref{Th:49}, without loss of generality, we can assume that $\Gamma_h^-$ is the generlized singular line. Indeed, the vanishing order of the eigenfunction in such a case can be obtained from formally taking $\eta_2$ on $\Gamma_h^+$ to be zero in Theorem \ref{th:41}. 

\begin{theorem}\label{th:410}
	Let $u$ be a Laplacian eigenfunction to \eqref{eq:eig}. Suppose that a singular line $\Gamma_h^+ \in {\mathcal S}^\lambda_{\Omega }$ intersects with a generalized  singular line $\Gamma_h^- \in {\mathcal M}^\lambda_{\Omega }$ at the origin with the angle $\angle (\Gamma_h^+,\Gamma_h^-)=\alpha\cdot \pi$. Assume that the boundary parameter $\eta_1$ on $\Gamma_h^-$ is a non-zero constant, i.e., $\eta_1 \equiv C_1 \neq 0$. Then  the  Laplacian eigenfunction $u$ vanishes up to the order $N$ at $\mathbf{0}$ with respect to $\angle (\Gamma_h^+,\Gamma_h^-)=\alpha\cdot \pi$:
	\begin{subequations}
		\begin{align*}
			&N= n,\quad \mbox{if }  u(\mathbf 0)  =0 \mbox{ and } \alpha \neq \frac{q  }{p}, \, p=1,\ldots, n-1,  
				\end{align*}
				where $n\in\mathbb{N}$, $n\geq 3$  and $ q=1,2,\ldots, p-1$ for a fixed $p$. 
	\end{subequations}
	\end{theorem}

Using a similar proof to Theorem \ref{Th:49}, we can find the relationship between the vanishing order of 
the Laplacian eigenfunction and the intersecting angle between a singular line and a nodal line.  

\begin{theorem}\label{Th:411}
	Let $u$ be a Laplacian eigenfunction to \eqref{eq:eig}.  Suppose that a singular line $\Gamma_h^+ \in {\mathcal S}^\lambda_{\Omega }$ intersects with a  nodal line $\Gamma_h^- \in  {\mathcal N}^\lambda_{\Omega }$ at the origin with the angle $\angle (\Gamma_h^+,\Gamma_h^-)=\alpha\cdot \pi$. Then  the  Laplacian eigenfunction $u$ vanishes up to the order $N$ at $\mathbf{0}$ with respect to $\angle (\Gamma_h^+,\Gamma_h^-)=\alpha\cdot \pi$:
	\begin{subequations}
		\begin{align*}
			&N= n,\quad \mbox{if } \alpha \neq \frac{2q +1 }{2p}  , \, p=1,\ldots, n-1,  
				\end{align*}
				where $n\in\mathbb{N}$, $n\geq 2$  and for a fixed $p$, $ q=0,1,\ldots, p-1.$
	\end{subequations}

	\end{theorem} 
	
	\begin{remark}\label{re:38}
	
	As mentioned after Theorem~\ref{th:41}, we exclude the special case that the intersecting angle 
	between two lines is $\pi/2$. In fact, for Theorems \ref{th:41}-\ref{th:48} and Theorem \ref{th:410}, one may  see from their proofs in Section \ref{sec:pro1} that if $\angle (\Gamma_h^+,\Gamma_h^-)=\pi/2$, then there holds that $\nabla u(\mathbf 0)  =0$ if $u(\mathbf 0) =0$. That means, the eigenfunction is vanishing at least to the second order in such a case. For the other two cases in Theorems \ref{Th:49} and \ref{Th:411}, we can only have that if $\alpha=1/2$ and $u(\mathbf 0)=0$, the eigenfunction is vanishing at least to the first order.

	\end{remark}
	
	\begin{remark}
	{It is noted that in Theorems \ref{Th:49} and \ref{Th:411}, we require that $n\geq 2$, whereas in other theorems, we require that $n\geq 3$. In particular, when $n=2$, $\alpha \neq 1/2$, one can conclude in Theorems \ref{Th:49} and \ref{Th:411} that the eigenfunction is vanishing at least to the second order. This conclusion is different from Theorems \ref{th:41}-\ref{th:48} and \ref{th:410}, where one has that if $\alpha \neq 1/2$ then the eigenfunction is vanishing at least to the third order.  }
	\end{remark}

	\section{Irrational intersection and infinite vanishing order: general cases}\label{sec:irrational gene}

In this section we consider the irrational intersection, namely $\alpha$ in \eqref{eq:angle1} is an irrational number. We show that the eigenfunction is generically vanishing to infinity, namely $u$ is identically zero in $\Omega$. Here, the generic condition is provided by $u$ vanishing or not at the intersecting point. We shall present more 
discussions on this generic condition in Section~\ref{sect:discussion}. From Theorems \ref{th:41}, \ref{Th:49}, \ref{th:410} and \ref{Th:411}, we have the following four theorems directly without proofs. 


\begin{theorem}
Let $u$ be a Laplacian eigenfunction to \eqref{eq:eig}. Suppose that there are two generalized singular lines $\Gamma_h^+$ and $\Gamma_h^-$ from ${\mathcal M}^\lambda_{\Omega }$ such that \eqref{eq:angle} and \eqref{eq:angle2a} hold. 
	Assume that $\eta_1 \equiv C_1$ and $\eta_2 \equiv  C_2$, where $C_1 $ and $C_2$ are two constants. If $\angle (\Gamma_h^+,\Gamma_h^-)=\alpha\cdot \pi$ with $\alpha\in (0, 2)$ irrational, then there hold that 
		\begin{align*}
		& \mathrm{Vani}(u; \mathbf{0})=0,\quad\ \ \ \, \mbox{if } u({\mathbf 0}) \neq 0;  \\
		&\mathrm{Vani}(u; \mathbf{0})= +\infty,\quad \mbox{if } u({\mathbf 0}) = 0.
		\end{align*}
	\end{theorem} 

{\begin{theorem}\label{Th:51 gene}
	Let $u$ be a Laplacian eigenfunction to \eqref{eq:eig}. Suppose that a generalized singular  line $\Gamma_h^- \in {\mathcal M}^\lambda_{\Omega }$ intersects with a nodal  line $\Gamma_h^+ \in {\mathcal N}^\lambda_{\Omega }$ at $\mathbf 0$ with the angle $\angle (\Gamma_h^+,\Gamma_h^-)=\alpha\cdot \pi$. Assume that the boundary parameter $\eta_1 \equiv C_1 $ on $\Gamma_h^-$ is a constant.  If $\alpha\in (0, 2)$ is irrational, then there holds 
		\begin{align*}
		\mathrm{Vani}(u; \mathbf{0})= +\infty.
		\end{align*}
	\end{theorem} 
	
	}

\begin{theorem}
	Let $u$ be a Laplacian eigenfunction to \eqref{eq:eig}. Suppose that a singular line $\Gamma_h^+ \in {\mathcal S}^\lambda_{\Omega }$ intersects with a generalized  singular line $\Gamma_h^- \in {\mathcal M}^\lambda_{\Omega }$  at $\mathbf 0$ with the angle $\angle (\Gamma_h^+,\Gamma_h^-)=\alpha\cdot \pi$. Assume that the boundary parameter $\eta_1 \equiv C_1 $ on $\Gamma_h^-$ is a constant.  If $\alpha\in (0, 2)$ is irrational, then there hold that 
		\begin{align*}
		& \mathrm{Vani}(u; \mathbf{0})=0,\quad\ \ \ \, \mbox{if } u({\mathbf 0}) \neq 0;  \\
		&\mathrm{Vani}(u; \mathbf{0})= +\infty,\quad \mbox{if } u({\mathbf 0}) = 0.
		\end{align*}
	\end{theorem}

{
\begin{theorem}

Let $u$ be a Laplacian eigenfunction to \eqref{eq:eig}. Suppose that a singular  line $\Gamma_h^- \in {\mathcal S}^\lambda_{\Omega }$ intersects with a nodal  line $\Gamma_h^+ \in {\mathcal N}^\lambda_{\Omega }$ at $\mathbf 0$ with the angle $\angle (\Gamma_h^+,\Gamma_h^-)=\alpha\cdot \pi$. If $\alpha\in (0, 2)$ is irrational, then there holds 
		\begin{align*}
		\mathrm{Vani}(u; \mathbf{0})= +\infty.
		\end{align*}
	\end{theorem} 
}

\section{Proofs of the theorems in Section~\ref{sec:3} up the third order}\label{sec:pro1}

In this section, we present the proofs of theorems shown in Section~\ref{sec:3}, but confined to the case that the vanishing order $N$ is at most 3. We develop a mathematical scheme by making use of tools from microlocal analysis that possesses several remarkable properties. 
Next, we first introduce the so-called complex-geometrical-optics (CGO) solutions constructed in \cite{Bsource} for the subsequent use. As before, we let $(r,\theta)$ denote the polar coordinates in $\R^2$; that is, for $\mathbf x=(x_1,x_2)\in \R^2$, one has $r=|\mathbf x|$ and $\theta=\mathrm { arg}(x_1 +\bsi x_2)$. Let $B_h$ be the central disk of radius $h\in\mathbb{R}_+$. Let $\Gamma^\pm$ signify the infinite extension of $\Gamma_h^\pm$ in the half-space $x_2\geq 0$. Let $\theta_m=0$ and $\theta_M\in (0,\pi)$ be respectively the polar angles of $\Gamma^-$ and $\Gamma^+$. Consider the open sector 
\begin{equation}\label{eq:W}
	 W=\Big \{ \mathbf{x}\in \R^2; \bfx\neq 0,\   \theta_m < {\rm arg}(x_1+ \bsi x_2 ) < \theta_M \Big \},
\end{equation}
which is formed by the two half-lines $\Gamma^-$ and $\Gamma^+$. We have the following result. 

\begin{lemma}\label{lem:1}\cite[Lemma 2.2 and Proposition 2.3]{Bsource}
Let
\begin{equation}\label{eq:cgo}
	u_0(\mathbf  x):= \exp \left( \sqrt r \left(\cos \left(\frac{\theta}{2}+\pi\right) +\bsi \sin \left(\frac{\theta}{2} +\pi\right)  \right ) \right) .
\end{equation}
Then $\Delta u_0=0$ in $\R^2\backslash (\R_-\times \{ \mathbf{ 0} \}\cup \{(0,0)\}) $, and $s \mapsto u_0(s\mathbf x) $ decays exponentially in $\R_+$ whenever $\mathbf x$ is in the same domain of harmonicity.  
\mm{
Furthermore, it holds for $\alpha, s >0$ that 
\begin{equation}\label{eq:xalpha}
	\int_W |u_0(s\mathbf{ x})| |\mathbf{x} |^\alpha {\rm d}\mathbf{ x} \leq \frac{2(\theta_M-\theta_m )\Gamma(2\alpha+4) }{ \delta_W^{2\alpha+4}} s^{-\alpha-2},
\end{equation}
where $\delta_W=-\max_{ \theta_m < \theta <\theta_M }  \cos(\theta/2+\pi ) >0$, and 
\begin{equation}\label{eq:u0w}
	\int_W u_0(s \bfx ) {\rm d} \bfx = 6 \bsi (e^{-2\theta_M \bsi }-e^{-2\theta_m \bsi }  ) s^{-2},
	\end{equation}
	while for $h\in\mathbb{R}_+$,
	\begin{equation}\label{eq:1.5}
	\int_{W \backslash B_h } |u_0(s\bfx)|   {\rm d} \bfx \leq \frac{6(\theta_M-\theta_m )}{\delta_W^4} s^{-2} e^{-\delta_W \sqrt{hs}/2}.
	\end{equation}
	}
\end{lemma}

%

Henceforth, for notational convenience, we set
\begin{equation}\label{eq:theta_0}
	\theta_0=\alpha\cdot \pi =\angle(\Gamma_h^+, \Gamma_h^- );
\end{equation}
for $\alpha\in(0,1)$. In order to make use of the CGO solution $u_0(s\bfx)$ given in Lemma \ref{lem:1} as a test function to analyze the vanishing order of $u$ at the origin, we consider the following domain (see Fig.~\ref{fig1} for the illustration):
\begin{equation}\label{eq:sh}
	S_h=W\cap B_h,
\end{equation}
where $\partial S_h=\Gamma_h^{+} \cup \Gamma_h^- \cup \Lambda_h$ and
\begin{equation}\label{eq:gammaEX}
	\begin{split}
		&\Gamma_h^{+}=\{ \bfx\in \R^2 ;~0\leq \sqrt{x_1^2+x_2^2}\leq h,\,  {\rm arg}(x_1+ \bsi x_2 )=\theta_0   \},
\\
&\Gamma_h^{-}=\{ \bfx\in \R^2 ;~0\leq \sqrt{x_1^2+x_2^2}\leq h,\,  {\rm arg}(x_1+ \bsi x_2 )=0 \},\\
		&\Lambda_h= W \cap \partial B_h.
	\end{split}
\end{equation}
In the definition of the generalized singular line, we recall that the polar angles of the exterior normal vectors  of $\Gamma^+_h$ (with respect to the domain $W$) and $\Gamma_h^-$ are respectively 
\begin{equation}
	\varphi_M=\theta_0+\frac{\pi}{2}, \quad \varphi_m=-\frac{\pi}{2}. 
\end{equation}

				 In order to investigate the relationship between the vanishing order of $u$ at the origin and the  the intersecting angle of the generalized singular lines $\Gamma_h^\pm$, we consider the following equations	
\begin{subequations}
\begin{align}
& \Delta u+\lambda u=0\hspace*{1.4cm} \mbox{in}\ \ B_h,\label{eq:helm2} \medskip\\
& \displaystyle{\frac{\partial u}{\partial \nu}+ \eta_1 u=0}\hspace*{1.3cm} \mbox{on}\ \  \Gamma_h^-  ,\label{eq:helm2bc1} \medskip\\  
& \displaystyle{\frac{\partial u}{\partial \nu}+ \eta_2 u=0  }\hspace*{1.3cm} \mbox{on}\ \  \Gamma_h^{+}. \label{eq:helm2bc2}
\end{align}
\end{subequations}
Next, we derive several crucial auxiliary results regarding the function $u$ satisfying \eqref{eq:helm2}--\eqref{eq:helm2bc2}.

Recall that $S_h$ is defined in \eqref{eq:sh}. Since $u_0$ is only smooth in $S_h\backslash \, B_\varepsilon (0<\varepsilon<h )$, we cannot use Green's formula for the Laplacian  eigenfunction $u$ and the CGO solution $u_0$ in $S_h$ directly. Instead, we may overcome this difficulty  by taking a limit of the volume integral with the  integrand $u_0 \Delta u$ over $S_h\backslash \, B_\varepsilon$  and carefully investigate the boundary integrals on $\partial (S_h\backslash \, B_\varepsilon)$. We only present the result in the following proposition and its proof is similar to the argument of Lemma 3.2 in \cite{Bsource}. 
\begin{lemma}\label{lem:aux1}
	The CGO solution $u_0(sx)$ defined in \eqref{eq:cgo} is  harmonic in $S_h \backslash\mathbf{0} $ and decays exponentially as $s \rightarrow \infty$ for $0<\theta<\theta_0$, where $\theta_0$ is the intersecting angle of $\Gamma_h^+$ and $\Gamma_h^-$. Moreover, for the Laplacian eigenfunction $u$ to \eqref{eq:eig},  
	the Green's formula holds 
	\begin{align} \label{eq:green}
\int_{S_h} \left (u_0(s\bfx)\Delta u- u \Delta u_0(s\bfx) \right ) {\rm d} \bfx
=I_1^++I_1^-+I_2, 
\end{align}
where
\begin{equation}\label{eq:Inotation}
	\begin{split}
		I_1^+&=\int_{ \Gamma_h^{+} } \left (  u_0(s\bfx) \frac{\partial u}{\partial \nu}-u(\bfx)   \frac{\partial u_0(s\bfx)}{\partial \nu} \right) {\rm d} \sigma, \\
		 I_1^-&=\int_{ \Gamma_h^- } \left (  u_0(s\bfx) \frac{\partial u}{\partial \nu}-u(\bfx)   \frac{\partial u_0(s\bfx)}{\partial \nu} \right) {\rm d} \sigma,  \\
	I_2&=\int_{ \Lambda_h } \left (  u_0(s\bfx) \frac{\partial u}{\partial \nu}-u(\bfx)   \frac{\partial u_0(s\bfx)}{\partial \nu} \right) {\rm d} \sigma. 
	\end{split}
\end{equation}
\end{lemma}

From Lemma \ref{lem:26}, by direct calculations, we have the following proposition regarding the exterior normal derivative of the CGO solution $u_0(s \mathbf x)$ on any straight line. 
\begin{proposition}\label{pro:3.7}
	For any straight line $\Gamma$, where $ \mathbf  x=r(\cos \theta, \sin \theta)\in \Gamma$, let the exterior unit normal vector to $\Gamma$ be  $\nu =(  \cos \varphi, \sin \varphi )$. Then the CGO solution $u_0(s\bfx)$ given in Lemma \ref{lem:1} fulfills
	\begin{align}\label{eq:u0derivative}
	\frac{\partial u_0(s\mathbf x) }{ \partial \nu}\Big|_{\Gamma}&=\beta(\theta)  e^{\sqrt{sr} \zeta(\theta) } \sqrt{\frac{s}{r}}, 
\end{align}
where 
\begin{align}\label{eq:beta}
\zeta(\theta)= e^{\bsi (\theta/2+\pi ) } =-e^{\bsi \theta/2},\, 
\beta(\theta)&=\frac{1}{2}\sin(\varphi-\theta )\zeta(\theta)',\, 
	\zeta(\theta)' =-\bsi e^{\bsi \theta/2 }. 
\end{align}
\end{proposition}

By induction and straightforward calculations, we can derive the explicit formulas of the following integrals in Lemma \ref{pro:int}, which is essential in showing the relationship between the  vanishing order of $u$ and  the intersecting angle of the generalized singular lines $\Gamma_h^\pm$. The detailed proof of Lemma \ref{pro:int} is omitted. 
\begin{lemma}
	\label{pro:int}
	For a given $\zeta(\theta)\in \C$ and $\ell=0,1,2,\ldots$, it holds that 
	\begin{equation*}
	\begin{split}
		\int_0^h  r^\ell   e^{ \sqrt{sr} \zeta(\theta)}  {\rm d } r& = \frac{2}{s^{\ell +1 }} \Big\{ \frac{(2 \ell +1)!}{ \zeta(\theta)^{2\ell +2 } } + e^{\sqrt{sh}   \zeta(\theta)  }\sum_{j=0}^{2\ell+1  } \frac{ (-1)^j (2\ell +1)! }{(2\ell+1 -j)!  \zeta(\theta)^{j+1} } (sh)^{(2\ell+1-j)/2}  \Big\},\\
		\int_0^h  r^\ell   e^{ \sqrt{sr} \zeta(\theta)} \sqrt{\frac{s}{r}} {\rm d } r& = \frac{2}{s^{\ell  }} \Big\{ -\frac{(2 \ell )!}{ \zeta(\theta)^{2\ell +1 } } + e^{\sqrt{sh}   \zeta(\theta)  }\sum_{j=0}^{2\ell  } \frac{ (-1)^j (2\ell )! }{(2\ell -j)!  \zeta(\theta)^{j+1} } (sh)^{(2\ell-j)/2}  \Big\}. 
	\end{split}
	\end{equation*}
\mm{
Furthermore, the following asymptotic expansions are true for $\Re(\zeta(\theta))<0$ and $s\rightarrow \infty$: 
}
\begin{equation}\label{eq:int import}
	\begin{split}
		\int_0^h  r^\ell   e^{ \sqrt{sr} \zeta(\theta)}  {\rm d } r& = \frac{2}{s^{\ell +1 }} \cdot  \frac{(2 \ell +1)!}{ \zeta(\theta)^{2\ell +2 } }   + \Oh\left(s^{-1/2 }e^{\sqrt{sh}   \zeta(\theta)  }\right) ,\\
		\int_0^h  r^\ell   e^{ \sqrt{sr} \zeta(\theta)} \sqrt{\frac{s}{r}} {\rm d } r& =-  \frac{2}{s^{\ell  }} \cdot   \frac{ (2 \ell )!}{ \zeta(\theta)^{2\ell +1 } } + \Oh\left(e^{\sqrt{sh}   \zeta(\theta)  }\right)\,.
	\end{split}
	\end{equation}
\end{lemma}

\mm{
In Lemmas \ref{pro:56} and \ref{pro:57} below, we will investigate the asymptotic behaviors of the integrals associated with $u$ and the CGO solution $u_0(s \mathbf x)$ (or their corresponding exterior normal derivatives) with respect to the positive parameter $s$ as $s \rightarrow \infty$. 
}

\begin{lemma}\label{pro:56}
	Recall that $\Gamma_h^-$ and $\Gamma_h^{+}$ are defined in \eqref{eq:gammaEX}. Denote
	\begin{equation}\label{eq:I11notaion}
		I_{11}^+=\int_{ \Gamma_h^{+} } u(\bfx)   \frac{\partial u_0(s\bfx)}{\partial \nu} {\rm d} \sigma,\quad I_{11}^-=\int_{ \Gamma_h^{- } }u(\bfx)   \frac{\partial u_0(s\bfx)}{\partial \nu} {\rm d} \sigma.
	\end{equation}
	\mm{
	Then the following asymptotic expansions hold with respect to $s$ as $s\rightarrow \infty$: 
	}
	\begin{align}
	\begin{split}
	I_{11}^+&=-\frac{2\beta(\theta_0) }{\zeta(\theta_0)  }u(\mathbf 0)	- \frac{1}{s} \cdot \frac{4\beta(\theta_0) }{\zeta(\theta_0)^3}c_1(\theta_0) -\frac{1}{s^2} \cdot \frac{48\beta(\theta_0) }{\zeta(\theta_0)^5}c_2(\theta_0) +\Oh(s^{-3}) , \\
	I_{11}^-&=-\frac{2\beta(0) }{\zeta( 0)  }u(\mathbf 0)	- \frac{1}{s} \cdot \frac{4\beta( 0) }{\zeta( 0)^3}c_1( 0) -\frac{1}{s^2} \cdot \frac{48\beta(0) }{\zeta(0)^5}c_2( 0) +\Oh(s^{-3}) ,\label{eq:asymptotic}
	\end{split}
	\end{align}
	where
	\begin{equation}\label{eq:citheta}
		\begin{split}
			c_1(\theta)&=\frac{\partial u}{\partial x_1}\Big|_{\bfx=\mathbf 0} \cos \theta+\frac{\partial u}{\partial x_2}\Big|_{\bfx=\mathbf 0} \sin \theta,\\
			c_2(\theta)&=\frac{1}{2}\left(\frac{\partial^2 u}{\partial x_1^2}\Big|_{\bfx=\mathbf0} \cos^2 \theta +\frac{\partial^2 u}{\partial x_1 x_2}\Big|_{\bfx=\mathbf 0} \sin 2\theta  + \frac{\partial^2 u}{\partial x_2^2}\Big|_{\bfx=\mathbf 0} \sin^2 \theta \right). 
		\end{split}
	\end{equation}
\end{lemma}
\begin{proof} It is easy to see that the exterior unit normal vector to $\Gamma_h^{+}$ is
\begin{equation}\label{eq:phiM}
	\nu=(\cos \varphi_M, \sin \varphi_M ), \quad \varphi_M=\theta_0+\frac{\pi}{2}. 
\end{equation}
From Proposition \ref{pro:3.7}, on $\Gamma_h^{+}$ we obtain that 
\begin{align}\label{eq:u0Gmma+}
	\frac{\partial u_0(s\bfx) }{ \partial \nu}\Big|_{\Gamma_h^{+}}&=\beta(\theta_0 )  e^{\sqrt{sr} \zeta(\theta_0) } \sqrt{\frac{s}{r}}, 
\end{align}
where $\zeta(\theta_0)= -e^{\bsi \theta_0/2 }$, and 
\begin{align}\label{eq:betaM}
\beta(\theta_0)&=\frac{1}{2}\sin(\varphi_M-\theta_0  )\zeta(\theta_0)'=-\frac{\bsi  e^{\bsi \theta_0/2}}{2},\quad
	\zeta(\theta_0)' =- \bsi e^{\bsi \theta_0/2 }. 
\end{align}
\mm{Noting the analyticity of the Laplacian eigenfunction $u$ to \eqref{eq:eig} in $\Omega$},  we have the expansion near a neighborhood of the origin:
\begin{equation}\label{eq:uex251}
	u(\mathbf{ x})=\sum_{ \alpha \in {\mathbb N}_0^2, \, |\alpha| \geq 0  }\frac{(\partial^\alpha u)(\mathbf{ 0}) }{\alpha ! }\mathbf{ x}^\alpha.  
\end{equation} 
Then substituting \eqref{eq:u0Gmma+} and \eqref{eq:uex251} into $I_{11}^+$, we deduce that
\begin{align}\label{eq:329}
\begin{split}
	I_{11}^+&=\int_{ \Gamma_h^{+} } u(\bfx)   \frac{\partial u_0(s\bfx)}{\partial \nu} {\rm d} \sigma=u(\mathbf 0) \beta(\theta_0 )\int_{0 }^h    e^{\sqrt{sr} \zeta(\theta_0) } \sqrt{\frac{s}{r}} {\rm d} r \\
	&+\left(\frac{\partial u}{\partial x_1}\Big|_{\bfx=\mathbf 0} \cos \theta_0+\frac{\partial u}{\partial x_2}\Big|_{\bfx=\mathbf 0} \sin \theta_0\right) \beta(\theta_0 )\int_{0 }^h r   e^{\sqrt{sr} \zeta(\theta_0) } \sqrt{\frac{s}{r}} {\rm d} r  \\
	&+\frac{1}{2}\left(\frac{\partial^2 u}{\partial x_1^2}\Big|_{\mathbf x=\mathbf 0} \cos^2\theta_0 +2\frac{\partial^2 u}{\partial x_1 x_2}\Big|_{\mathbf x=\mathbf 0} \sin\theta_0 \cos\theta_0 + \frac{\partial^2 u}{\partial x_2^2}\Big|_{\mathbf x=\mathbf 0} \sin^2 \theta_0\right) \\
	&\times  \beta(\theta_0 )\int_{0 }^h r^2   e^{\sqrt{sr} \zeta(\theta_0) } \sqrt{\frac{s}{r}} {\rm d} r  +r_{I_{11}^+},
\end{split}
\end{align}
where 
$$
r_{I_{11}^+}= \sum_{ \alpha \in {\mathbb N}_0^2, \, |\alpha| \geq 3  }\frac{(\partial^\alpha u)(\mathbf 0) }{\alpha ! }\int_{ \Gamma_h^{+} }  \mathbf x^\alpha \frac{\partial u_0(s\mathbf x)}{\partial \nu}  {\rm d} \sigma. 
$$
\mm{But this term can be estimated as $s\rightarrow \infty$ by means of \eqref{eq:int import}, }
\begin{align*}
	\left| r_{I_{11}^+} \right| &\leq |\beta(\theta_0 )|  \int_{0 }^h   r^3 \cdot  e^{\sqrt{sr} \Re(\zeta(\theta_0 ) ) } \sqrt{\frac{s}{r}}  {\rm d} r \sum_{ \alpha \in {\mathbb N}_0^2, \, |\alpha| \geq 3  } h^{|\alpha|-3}\left | \frac{(\partial^\alpha u)( \mathbf  0) }{\alpha ! } \right| =\Oh(s^{-3})\,,
	\end{align*}
	\mm{where we have used the fact that $\Re(\zeta(\theta_0) )=-\cos(\theta_0/2)<0$ for $\theta_0\in (0,\pi)$. Now the asymptotic expansion \eqref{eq:asymptotic} for $I_{11}^+$ follows directly by 
substituting \eqref{eq:int import} into \eqref{eq:329}. 
	Similar argument also applies to the  asymptotic expansion of $I_{11}^-$ and the detail is omitted.}
	
	The proof is complete. 
\end{proof}

\begin{lemma}\label{pro:57}
	Recall that $\Gamma_h^-$ and $\Gamma_h^+$ are defined in \eqref{eq:gammaEX}. Suppose that $u$ satisfies the boundary conditions  \eqref{eq:helm2bc1} and  \eqref{eq:helm2bc2} on  $\Gamma_h^-$ and $\Gamma_h^{+}$  respectively.  Moreover, assume that $ \eta_2 \in C^\gamma ( \Gamma_h^{+} )$  and $\eta_1 \in C^\gamma (\Gamma_h^-)$ for $\gamma\in (0, 1]$, and 
	\begin{equation}\label{eq:I12notation}
		I_{12}^+=- \int_{ \Gamma_h^{+} }u_0(s\mathbf x) \frac{\partial u}{\partial \nu}  {\rm d} \sigma,\quad I_{12}^-=- \int_{ \Gamma_h^{- } }u_0(s\mathbf x) \frac{\partial u}{\partial \nu}  {\rm d} \sigma.
	\end{equation}
	\mm{
	Then the following asymptotic expansions hold for $I_{12}^\pm$ with respect to $s$ as $s\rightarrow \infty$: 
	}
	\begin{align}\label{eq:asymptotic12}
	\begin{split}
		I_{12}^+&=\frac{2}{s} \cdot \frac{\eta_2( \mathbf 0) u( \mathbf  0) }{\zeta(\theta_0)^2} +\frac{12}{s^2} \cdot \frac{\eta_2( \mathbf 0) c_1( \theta_0) }{\zeta(\theta_0)^4}+u(\mathbf 0)\cdot  \Oh(s^{-1-\gamma})  +\Oh(s^{-2-\gamma }) , \\
	I_{12}^-&=\frac{2}{s} \cdot \frac{\eta_1( \mathbf 0) u( \mathbf 0) }{\zeta(0)^2} +\frac{12}{s^2} \cdot \frac{\eta_1( \mathbf 0) c_1( 0) }{\zeta(0)^4}+u(\mathbf 0)\cdot  \Oh(s^{-1-\gamma})  +\Oh(s^{-2-\gamma }) . 
	\end{split}
	\end{align}
\end{lemma}
\begin{proof}
\mm{Since $\eta_1$ and $\eta_2$ are of $C^\gamma$-smooth}, we have
\begin{equation}\label{eq:eta335}
 \eta_i(x)=\eta_i (\mathbf 0)+\delta \eta_i(x),  \quad |\delta  \eta_i |\leq \| \eta_i \|_{C^\gamma } \cdot |x|^\gamma. 
\end{equation}
Then using \eqref{eq:int import}, \eqref{eq:uex251} and \eqref{eq:eta335}, we can deduce in polar coordinates on $\Gamma_h^{+}$ that 
	\begin{align}\label{eq:I12+33}
	\begin{split}
		I_{12}^+= \int_{ \Gamma_h^{+} }u_0(s\mathbf x)  \eta_2  u  {\rm d} \sigma&=\eta_2(\mathbf 0) u(\mathbf 0) \int_{ 0 }^h  e^{ \sqrt{sr} \zeta(\theta_0 )}  {\rm d} r+u(\mathbf 0) \int_{ 0 }^h  \delta \eta_2  e^{ \sqrt{sr} \zeta(\theta_0 )}  {\rm d} r\\
		&+ \eta_2(\mathbf 0)c_1(\theta_0) \int_{ 0 }^h  r e^{ \sqrt{sr} \zeta(\theta_0 )}  {\rm d} r + r_{I_{12}^+},
	\end{split}
\end{align}
	where
	\begin{align}\label{eq:532r}
		 r_{I_{12}^+}& =  \eta_2(\mathbf 0)\sum_{ \alpha \in {\mathbb N}_0^2, \, |\alpha| \geq 2  }\frac{(\partial^\alpha u)(\mathbf 0) }{\alpha ! }   \int_{ \Gamma_h^{+} }u_0(s\mathbf x)\mathbf x^\alpha
   {\rm d} \sigma  +\sum_{ \alpha \in {\mathbb N}_0^2, \, |\alpha| \geq 1 }\frac{(\partial^\alpha u)(\mathbf 0) }{\alpha ! }   \int_{ \Gamma_h^{+} }u_0(s\mathbf x)\mathbf x^\alpha \delta  \eta_2
   {\rm d} \sigma.
  \notag
	\end{align}
\mm{For this term, it follows from \eqref{eq:int import} that }	
\begin{equation}\label{eq:334r12}
	\begin{split}
		\left| r_{I_{12}^+}  \right | &\leq  |\eta_2(\mathbf 0)| \sum_{ \alpha \in {\mathbb N}_0^2, \, |\alpha| \geq 2  } h^{|\alpha|-2}\left| \frac{(\partial^\alpha u)(\mathbf  0) }{\alpha ! }\right|    \int_{ 0 }^h  r^2  e^{ \sqrt{sr} \Re(\zeta(\theta_0-\pi )) } 
   {\rm d} r \\
   & +\| \eta_2 \|_{C^\gamma}  \sum_{ \alpha \in {\mathbb N}_0^2, \, |\alpha| \geq 1  }h^{|\alpha|-1 } \left| \frac{(\partial^\alpha u)(\mathbf  0) }{\alpha ! }  \right|  \int_{ 0}^h   r^{1+\gamma}  e^{ \sqrt{sr} \Re(\zeta(\theta_0-\pi )) }  {\rm d} r \\
   &=\Oh(s^{-2-\gamma}). 
   	\end{split}
	\end{equation}
\mm{Using this and \eqref{eq:int import} again, we can derive \eqref{eq:asymptotic12} from \eqref{eq:I12+33}.  
Similar derivation can be done for the asymptotic expansion of $I_{12}^-$. }
%
\end{proof}

The following lemma is about the exterior normal derivative of $\partial_\nu u$ with respect to any singular line of a Laplacian eigenfunction $u$. 
\begin{lemma}\label{pro:312}
	\mm{
	Suppose $\Gamma:=\{\mathbf  x\in \R^2 ~| ~\mathbf  x=r(\cos \theta, \sin \theta ), r>0\}$ ($\theta$ is fixed) is a singular line of the Laplacian eigenfunction $u$, and $\varphi$ is the polar angle of the unit normal vector to 
	$\Gamma$, then
	}
	\begin{equation}\label{eq:pro312}
		\cos \varphi \cos \theta  \frac{ \partial^2 u } {\partial x_1^2} \Big|_{\mathbf x=\mathbf 0} + \sin \varphi \sin  \theta \frac{ \partial^2 u } {\partial x_2^2} \Big|_{\mathbf x=\mathbf 0}+  \sin (\varphi +\theta)   \frac{ \partial^2 u } {\partial x_1 \partial x_2 } \Big|_{\mathbf x=\mathbf 0}  =0\,.
	\end{equation}
\end{lemma}
\begin{proof}
	\mm{
	Recalling the definition of a singular line, and using the fact that
	\begin{equation}\label{eq:a2}
	\nabla \left( \frac{\partial u}{\partial  \nu } \right) \cdot (\cos \theta ,\sin \theta)^\top =0, 
	\end{equation}
		we can derive \eqref{eq:pro312} by evaluating \eqref{eq:a2} in more detail at $\mathbf x=\mathbf 0$. 
		}
\end{proof}


For the subsequent analysis, we also need the following lemma. 

\begin{lemma}\label{prop710}
	Suppose that $u$ has the expansion \eqref{eq:uex251}. For any straight line segment $\Gamma_h:=\{x\in \R^2~|~\mathbf  x=r(\cos \theta, \sin \theta ), 0\leq r\leq h \ll 1\} $ (with $\theta$ fixed) satisfying $\Re(\zeta(\theta))<0$ where $\zeta(\theta )$ is defined in \eqref{eq:beta}, we assume that $\nu =(\cos \varphi, \sin \varphi ) $ is the exterior unit normal vector to $\Gamma_h$. Then it holds as $s\rightarrow \infty$ that 
	\begin{align}\label{eq:373int}
		I_{12}&= \int_{\Gamma_h} u_0(s\mathbf x) \frac{\partial u}{\partial \nu }  {\rm d} \sigma=\frac{2}{s} \cdot \left(\frac{\partial u}{\partial x_1} \Big |_{\mathbf x=\mathbf 0} \cos \varphi+  \frac{\partial u}{\partial x_2} \Big |_{\mathbf x=\mathbf 0} \sin  \varphi   \right)  \cdot  \frac{1}{ \zeta(\theta)^{2 } }  \\
		&+\frac{12}{s^{2 }} \cdot     \Bigg( \frac{\partial^2 u}{\partial x_1^2} \Big |_{\mathbf x=\mathbf 0} \cos \varphi \cos \theta+  \frac{\partial^2 u}{\partial x_2^2} \Big |_{\mathbf x=\mathbf 0} \sin  \varphi \sin \theta + \frac{\partial^2 u}{\partial x_1 \partial x_2} \Big |_{\mathbf x=\mathbf 0} \sin  (\varphi + \theta ) \Bigg  )  \cdot  \frac{1}{ \zeta(\theta)^{4 } } \notag\\
		&+\Oh(s^{-3} )\,.\notag
	\end{align}
\end{lemma}
\begin{proof}
	Using  the polar coordinate on $\Gamma_h$, it is easy to see that 
	\begin{align}
		\frac{\partial u}{\partial \nu } &= \frac{\partial u}{\partial x_1} \Big |_{\mathbf x=\mathbf 0} \cos \varphi+  \frac{\partial u}{\partial x_2} \Big |_{\mathbf x=\mathbf 0} \sin  \varphi + r \Bigg( \frac{\partial^2 u}{\partial x_1^2} \Big |_{\mathbf x=\mathbf 0} \cos \varphi \cos \theta+  \frac{\partial^2 u}{\partial x_2^2} \Big |_{\mathbf x=\mathbf 0} \sin  \varphi \sin \theta \notag \\
		&+ \frac{\partial^2 u}{\partial x_1 \partial x_2} \Big |_{\mathbf x=\mathbf 0} \sin  (\varphi + \theta ) \Bigg  ) +R(u,r,\varphi, \theta)
	\end{align}
	where 
	$$
	|R(u,r,\varphi, \theta)| \leq r^2 \sum_{ \alpha \in {\mathbb N}_0^2, \, |\alpha| \geq 3  } h^{|\alpha|-3}\left| \frac{(\partial^\alpha u)(\mathbf 0) }{(\alpha-1) ! }\right|. 
	$$
	\mm{Therefore we can further write}
	\begin{align}
	\begin{split}
		I_{12}&= \int_{\Gamma_h} u_0(s\mathbf x) \frac{\partial u}{\partial \nu }  {\rm d} \sigma=\left(\frac{\partial u}{\partial x_1} \Big |_{\mathbf x=\mathbf 0} \cos \varphi+  \frac{\partial u}{\partial x_2} \Big |_{\mathbf x=\mathbf 0} \sin  \varphi   \right) \int_0^h e^{\sqrt{sr} \zeta(\theta)} {\rm d} r \\
		&+  \Bigg( \frac{\partial^2 u}{\partial x_1^2} \Big |_{\mathbf x=\mathbf 0} \cos \varphi \cos \theta+  \frac{\partial^2 u}{\partial x_2^2} \Big |_{\mathbf x=\mathbf 0} \sin  \varphi \sin \theta + \frac{\partial^2 u}{\partial x_1 \partial x_2} \Big |_{\mathbf x=\mathbf 0} \sin  (\varphi + \theta ) \Bigg  )  \\
		&\times \int_0^h r e^{\sqrt{sr} \zeta(\theta)} {\rm d} r+ \int_0^h R(u,r,\varphi, \theta) e^{\sqrt{sr} \zeta(\theta)} {\rm d} r. 
	\end{split}
		\end{align}
\mm{Then the desired result follows from the estimate as $s\rightarrow \infty$ by using \eqref{eq:int import}:}
	\begin{equation}\notag
		\left | 	\int_0^h R(u,r,\varphi, \theta) e^{\sqrt{sr} \zeta(\theta)} {\rm d} r \right | \leq   \int_0^h r^2 e^{\sqrt{sr} \zeta(\theta)} {\rm d} r \cdot \sum_{ \alpha \in {\mathbb N}_0^2, \, |\alpha| \geq 3  } h^{|\alpha|-2}\left| \frac{(\partial^\alpha u)(\mathbf 0) }{(\alpha-1) ! }\right| =\Oh(s^{-3})\,.
	\end{equation}
%
\end{proof}

Suppose that $\Gamma_h$ is a nodal line of $u$. Using the polar coordinate and evaluating  \eqref{eq:uex251} on $\Gamma_h$, we can prove the following lemma. 
\begin{lemma}\label{pro:316}
	Suppose that $u$ is a Laplacian eigenfunction to \eqref{eq:eig} and 
	$$
	u\equiv 0 \mbox{ on } \Gamma_h,
	$$
	\mm{
	where $\Gamma_h:=\{\mathbf x\in \R^2 ;~ x=r(\cos \theta, \sin \theta), 0\leq r \leq h\}$ (with $\theta$ fixed ) 
	is a line segment. Then the functions $c_1(\theta)$ and $c_2(\theta)$ defined in \eqref{eq:citheta} are 
	both identically zero. 
	}
	\end{lemma}
%

Now we are in a position to present the proof of the theorems in Section~\ref{sec:3} in the specific case that the vanishing order $N$ is up to 3. Before that, we make two important remarks. 

First, throughout the present section, if a generalized singular line of the form \eqref{eq:normal1} is involved, the parameter $\eta$ can be a $C^1$ function other than a constant. Indeed, our argument in the present section can deal with this more general case. In principle, we believe that the theorems in Section~\ref{sec:3} can also be extended to the more general case that $\eta$ is a function other than a constant. However, when dealing with higher vanishing orders, the corresponding analysis becomes radically more tedious and complicated. Hence, in the next section for the general vanishing order case, we shall stick to the case that $\eta$ is a constant. 

Second, in proving the vanishing order of the eigenfunction (up to 3), we shall only make use of the eigenfunction confined in $S_h$. This is achieved by means of of the auxiliary results established in Lemmas~\ref{lem:aux1}--\ref{prop710}. It is emphasized that this is in sharp contrast to the argument in the next section (for the higher vanishing orders), which applies the spherical wave expansion of the eigenfunction in $B_h$. Hence, our argument in the present section is ``localized". This ``localization" property enables one to consider, e.g., the quantitative behaviours of the Dirichlet eigenfunction in $\Omega$ up to the boundary $\partial\Omega$. Moreover, combining with our first remark above, the argument can also be used to study the quantitative behaviours of eigenfunctions associated with a general second order elliptic operator other than the Laplacian. We shall investigate these interesting extensions in our future work. 

We first deal with Theorem~\ref{th:41}. According to our discussion above, we actually prove the following more general theorem. 
\begin{theorem}\label{th:311}
	Let $u$ be a Laplacian eigenfunction to \eqref{eq:eig}. Suppose that there are two generalized singular lines $\Gamma_h^+$ and $\Gamma_h^-$ from ${\mathcal M}^\lambda_{\Omega }$ such that \eqref{eq:angle} and \eqref{eq:angle2a} hold. 
	Assume that $\eta_1 \in C^1 ({ \Gamma^-_h} )$ and $\eta_2 \in C^{1}({ \Gamma^+_h} )$.  
	\mm{If the conditions 
		\begin{align}\label{eq:342b}
			 u(\mathbf{0}) =0 \mbox{ and } \alpha\neq \frac{1 }{2} 
		\end{align}
are satisfied, the Laplacian eigenfunction $u$ vanishes up to the order $3$ at $\mathbf{0}$.}
\end{theorem}

\begin{proof}

Recall Figure \ref{fig1}.  Evaluating \eqref{eq:helm2bc1} and \eqref{eq:helm2bc2} at $\mathbf 0$, using 
$ u(\mathbf 0 )=0$ we derive 
$$
\nabla u \Big |_{\mathbf x=\mathbf 0 } \cdot (\cos\varphi_m ,\sin \varphi_m )  =- \eta_1 ( \mathbf 0 ) u(\mathbf 0 ) =0,\quad \nabla u \Big |_{\mathbf x=\mathbf 0 } \cdot (\cos\varphi_M ,\sin \varphi_M )  =- \eta_2 ( \mathbf 0 ) u(\mathbf 0 ) =0. 
$$
Since $\theta_0\in (0, \pi)$, we know that  $\Gamma_h^-$ and $\Gamma_h^+$ are non-collinear. Therefore it is easy to see
\begin{equation}\label{eq:derit zero}
	\nabla u (\mathbf 0)  =0.
\end{equation}
\mm{Noting that $u$ is the Laplacian eigenfunction satisfying \eqref{eq:helm2}, we derive the 
integral equality from \eqref{eq:green}:}
\begin{align} \label{eq:int1}
-\lambda	\int_{S_h} u_0(s\mathbf x) u(\mathbf x) {\rm d} \mathbf x &=\int_{S_h} \left (u_0(s\mathbf x)\Delta u- u \Delta u_0(s\mathbf x) \right ) {\rm d} \mathbf x
=I_1^++I_1^-+I_2, 
\end{align}
where $S_h$ is defined in \eqref{eq:sh}, and $I_1^\pm$ and $I_2$ are defined in \eqref{eq:Inotation}. 

Since $u\in H^2(B_h)$, which can be embedded into $C^\gamma (B_h) (0<\gamma<1)$, we know that
\begin{equation}\label{eq:u ex}
	u(\mathbf x)=u(\mathbf 0)+\delta u(\mathbf x),\quad | \delta u(\mathbf  x) | \leq \|u \|_{C^\gamma } |\mathbf x|^\gamma\,, 
\end{equation}
\mm{
from which it follows that
\begin{equation}\label{eq:int2}
\begin{split}
	\int_{S_h} u_0(s\mathbf x) u(\mathbf x) {\rm d} \mathbf x &= u(\mathbf 0)	\int_{S_h} u_0(s \mathbf  x) {\rm d}  \mathbf  x+ 	I_3
		=u(\mathbf 0) \left(\int_{W} u_0(s\mathbf x) {\rm d} \mathbf x -I_4 \right )  +I_3 
		\end{split}
\end{equation}
}
where
\begin{align*}
	I_3&=\int_{S_h} u_0(s\mathbf x) \delta u(\mathbf x) {\rm d}\mathbf  x,\quad I_4= \int_{W \backslash B_h} u_0(s\mathbf x) {\rm d} \mathbf x. 
\end{align*}
Substituting \eqref{eq:int2} into  \eqref{eq:int1} and combining with \eqref{eq:u0w}, we derive that
\begin{align}\label{eq:int3r}
	-6 \lambda \bsi (e^{-2\bsi\theta_0} -1  ) s^{-2} u(\mathbf 0)&=I_1^++I_1^{-}+I_2+ \lambda I_3-  u(\mathbf 0) \lambda I_4. 
\end{align}
Using $u(\mathbf 0)=0 $, we further obtain 
\begin{align}\label{eq:int3}
	0&=I_1^++I_1^{-}+I_2+ \lambda I_3. 
\end{align}
Since $u \in H^2(B_h )$, 
\mm{
we have from \cite[Page 6263]{Bsource}  and \eqref{eq:xalpha} that for some 
$c'>0$, 
\begin{equation}\label{212}
	\left| I_2 \right| \leq C e^{-c' \sqrt s}, \quad \left| I_3 \right| \leq \Oh(s^{-\alpha-2 }) 
	\quad \mbox{as} ~~s \rightarrow \infty\,.
	\end{equation}
	}

Recalling the definitions of $I_1^\pm, I_2$, $I_{11}^\pm\mbox{ and } I_{12}^\pm$ given by \eqref{eq:Inotation}, \eqref{eq:I11notaion} and \eqref{eq:I12notation} respectively, it is easy to see that
\begin{equation}\label{eq:I1I2combine}
	I_1^+=-(I_{11}^++I_{12}^+),\quad I_1^-=-(I_{11}^-+I_{12}^-).
\end{equation}
Since $\eta_1$ and $\eta_2$ are $C^1$ functions on the boundary $\Gamma_h^\pm$, 
they fulfill the requirement of Lemma \ref{pro:57}. Using \eqref{eq:derit zero} and recalling $c_1(\theta_0)$ and $c_1(0)$ given in \eqref{eq:citheta}, we know that
\begin{equation}\label{eq:c zeoro}
	c_1(\theta_0)=c_1(0)=0.
\end{equation}
Substituting $u(\mathbf 0)=0$  and \eqref{eq:c zeoro} into \eqref{eq:asymptotic} and \eqref{eq:asymptotic12} yields 
\begin{align}\label{eq:asymptoticNEW}
\begin{split}
	I_{11}^+&= -\frac{1}{s^2} \cdot \frac{48\beta(\theta_0) }{\zeta(\theta_0)^5}c_2(\theta_0) +\Oh(s^{-3}) , \quad
	I_{11}^-=-\frac{1}{s^2} \cdot \frac{48\beta(0) }{\zeta(0)^5}c_2( 0) +\Oh(s^{-3}) ,\\
	I_{12}^+&= \Oh(s^{-2-\gamma }) ,\quad 	I_{12}^- = \Oh(s^{-2-\gamma }) . 
\end{split}
	\end{align}
\mm{
But by means of \eqref{eq:I1I2combine}, we deduce by substituting \eqref{eq:asymptoticNEW}  
into \eqref{eq:int3} that
}
\begin{align}\label{eq:int356}
	\frac{1}{s^2} \cdot \frac{48\beta(\theta_0) }{\zeta(\theta_0)^5}c_2(\theta_0)+ \frac{1}{s^2} \cdot \frac{48\beta(0) }{\zeta(0)^5}c_2(0) =-\lambda I_3 -I_2 +\Oh (s^{-2-\gamma}),
\end{align}
where $c_2(\theta ) $ is defined in \eqref{eq:citheta}. Multiplying $s^2$ on the both sides of \eqref{eq:int356}, combining with \eqref{212}, and letting $s\rightarrow \infty$, we can obtain that
\begin{equation}\label{eq:usecond357}
	\frac{\beta(\theta_0) }{\zeta(\theta_0)^5}c_2(\theta_0)+\frac{\beta(0) }{\zeta(0)^5}c_2(0)=0. 
\end{equation}
Using the eigen-equation, $-\Delta u=\lambda u$, we can see 
\begin{equation}\label{eq:2582nd}
	\frac{ \partial^2 u } {\partial x_1^2} \Big |_{\mathbf x=\mathbf 0}+ \frac{ \partial^2 u } {\partial x_2^2} \Big |_{\mathbf x=\mathbf 0}=-\lambda u(\mathbf 0)=0.
\end{equation}
Substituting this equation into the expression of $c_2(\theta)$, we derive that
\begin{equation}\label{eq:358eq}
	c_2(0)=\frac{1}{2}\frac{\partial^2 u}{\partial x_1^2}\Big|_{\mathbf x=\mathbf 0} ,\quad  c_2(\theta_0 )=\frac{1}{2}\left(\frac{\partial^2 u}{\partial x_1^2}\Big|_{\mathbf x=\mathbf 0} \cos 2\theta_0  +\frac{\partial^2 u}{\partial x_1 x_2}\Big|_{\mathbf x=\mathbf 0} \sin 2\theta_0   \right). 
\end{equation}
From \eqref{eq:beta}, it is easy to see that
\begin{equation}\label{eq:359eq}
	\frac{\beta(\theta_0) }{\zeta(\theta_0)^5}=\frac{\bsi e^{-2\bsi \theta_0 } }{2},\quad 	\frac{\beta(0) }{\zeta(0)^5}=-\frac{\bsi  }{2}
\end{equation}
Then substituting \eqref{eq:358eq} and \eqref{eq:359eq} into \eqref{eq:usecond357}, we can deduce that
\begin{equation}\label{eq:360}
	\left(1-e^{-2\bsi \theta_0}\cos 2 \theta_0  \right ) \frac{\partial^2 u}{\partial x_1^2}\Big|_{\mathbf x=\mathbf 0}  - e^{-2\bsi \theta_0}\sin  2 \theta_0  \frac{\partial^2 u}{\partial x_1 x_2}\Big|_{\mathbf x=\mathbf   0} =0. 
\end{equation}
\mm{
Recalling that 
$$
\displaystyle{ f(x):= \frac{\partial u}{\partial \nu}+\eta_2 u \equiv 0}\hspace*{1.3cm} \mbox{on}\ \  \Gamma_h^+   ,
 $$
we know the directional derivative of $f$ with respect to the direction $(\cos \theta_0, \sin \theta_0)$ satisfies
}
\begin{equation}\label{eq:na255}
	\nabla \left(\frac{\partial u}{\partial \nu}+\eta_2  u \right) \cdot (\cos \theta_0, \sin \theta_0)^\top =0. 
\end{equation}
\mm{
Since  $\eta_1 \in C^1 ({ \Gamma^-_h} )$ and $\eta_2 \in C^{1}({ \Gamma^+_h} )$, we can use 
the fact that 
$$
\nabla \left(\frac{\partial u}{\partial \nu}+\eta_2  u \right) =\begin{bmatrix}
\displaystyle{ 	\frac{ \partial^2 u } {\partial x_1^2} \cos \varphi_M + \frac{ \partial^2 u } {\partial x_1 \partial x_2 } \sin \varphi_M +\frac{ \partial \eta_2  } {\partial x_1 } u + \eta_2  \frac{ \partial u   } {\partial x_1 } } \\
\displaystyle{ 	\frac{ \partial^2 u } {\partial x_1 \partial x_2 } \cos \varphi_M + \frac{ \partial^2 u } { \partial x_2^2  } \sin \varphi_M +\frac{ \partial \eta_2   } {\partial x_2 } u + \eta_2  \frac{ \partial u   } {\partial x_2 } } 
\end{bmatrix}, 
$$
where $\varphi_M=\theta_0+\frac{\pi}{2}$,  and evaluate \eqref{eq:na255} at $ \mathbf x= \mathbf 0$, 
then derive by using $u(\mathbf 0 )=0$ and \eqref{eq:derit zero} that
}
\begin{equation}\label{eq:a3}
  \cos \varphi_M \cos \theta_0 \frac{ \partial^2 u } {\partial x_1^2} \Big|_{ \mathbf x= \mathbf 0} + \sin \varphi_M \sin  \theta_0 \frac{ \partial^2 u } {\partial x_2^2} \Big|_{ \mathbf x= \mathbf 0}+  \sin (\varphi_M +\theta_0)   \frac{ \partial^2 u } {\partial x_1 \partial x_2 }\Big|_{ \mathbf x= \mathbf 0}  =0.  
\end{equation}
Substituting \eqref{eq:2582nd} into \eqref{eq:a3}, together with $\varphi_M=\theta_0+\frac{\pi}{2}$  we can 
further obtain that
\begin{equation}\label{eq:363}
  \sin  2 \theta_0  \frac{ \partial^2 u } {\partial x_1^2} \Big|_{\mathbf x=\mathbf 0} -  \cos 2 \theta_0   \frac{ \partial^2 u } {\partial x_1 \partial x_2 } \Big|_{\mathbf x=\mathbf 0} =0.  
\end{equation}
\mm{
Now combining \eqref{eq:360} with \eqref{eq:363}, we can get a system of linear equations with respect to $\frac{\partial^2u}{\partial x_1^2}(\mathbf  0)$ and $\frac{\partial^2u}{\partial x_1\partial x_2}(\mathbf  0)$, 
with the determinant of its coefficient matrix given by 
$$
\left| \begin{array}{cc}
1-e^{-2\bsi \theta_0}\cos 2 \theta_0& - e^{-2\bsi \theta_0}\sin  2 \theta_0\\
	\sin  2 \theta_0 &-\cos   2 \theta_0 
\end{array} \right| =-\cos 2\theta_0+e^{-2\bsi \theta_0}=-\bsi \sin 2\theta_0 \neq 0
$$ 
since $\theta_0\neq \pi/2$. 
}
Therefore, together with \eqref{eq:2582nd} we can conclude that
$$
\frac{\partial^2 u}{\partial x_1^2}\Big|_{x=0}=\frac{ \partial^2 u } {\partial x_1 \partial x_2 } \Big|_{\mathbf x=\mathbf 0}=\frac{\partial^2 u}{\partial x_2^2}\Big|_{\mathbf x=\mathbf 0} =0. 
$$
Since the order of the lowest nontrivial homogeneous polynomial  in Taylor expansion \eqref{eq:uex251} 
around the origin is larger than 2, its vanishing order is at least up to 3. This completes the proof. 
\end{proof}

{{ We next deal with Theorem \ref{Th:49}, but under a more general situation with $\eta_2 \in C^{1}({ \Gamma^+_h})$}.}

\begin{theorem}\label{th:312}
	Let $u$ be a Laplacian eigenfunction to \eqref{eq:eig}. Suppose that a generalized singular line $\Gamma_h^+\in\mathcal{M}^\lambda_{\Omega }$  intersects with a nodal line $\Gamma_h^-\in{\mathcal N}^\lambda_{\Omega }$ such that \eqref{eq:angle} and \eqref{eq:angle2a} hold, and 
	$\eta_2 \in C^{1}({ \Gamma^+_h} )$.  If the following condition is fulfilled
	\begin{align}\label{eq:342b}
	\alpha\neq \frac{1 }{4}, \frac{1}{2} \mbox{ and } \frac{3}{4}
	\end{align}
	then the Laplacian eigenfunction $u$ vanishes up to the order $3$ at $\mathbf{0}$.
\end{theorem}

\begin{proof}
	\mm{
	Since 
	$ 
	u\equiv  0  \mbox{ on }   \Gamma_h^-, 
	$
	we know from \eqref{eq:2582nd} and Proposition \ref{pro:316} that 
	}
	\begin{equation}\label{eq:383}
	\frac{\partial u}{\partial x_1} \Big |_{\mathbf x=\mathbf 0} =\frac{ \partial^2 u } {\partial x_1^2} \Big |_{\mathbf x=\mathbf 0} = \frac{ \partial^2 u } {\partial x_2^2} \Big |_{\mathbf x=\mathbf 0}=0.
	\end{equation}
	\mm{
	Further, we derive by evaluating 
	\begin{equation*}
	\frac{\partial u}{\partial \nu } +\eta_2 u =0
	\end{equation*}
	on $\Gamma_h^+$ at $\mathbf x=\mathbf 0$ that }
	\begin{equation}\label{eq:a5}
	\frac{\partial u}{\partial x_1} \Big |_{\mathbf x=\mathbf 0} \cos \varphi_M + \frac{\partial u}{\partial x_2} \Big |_{\mathbf x=\mathbf 0}\sin \varphi_M=-\eta_2(\mathbf 0) u(\mathbf 0)=0,
	\end{equation}
	where $\varphi_M=\theta_0+\pi/2$. Then substituting \eqref{eq:383} into \eqref{eq:a5}, it is easy to see that
	$$
	\cos \theta_0 \cdot \frac{\partial u}{\partial x_2} \Big |_{\mathbf x=\mathbf 0}=0. 
	$$
	Hence if $\theta_0\neq \pi/2$, we have
	\begin{equation}\label{eq:386}
	\frac{\partial u}{\partial x_2} \Big |_{\mathbf x=\mathbf 0}=0. 
	\end{equation}
\mm{Now recalling that we have the boundary condition \eqref{eq:helm2bc2} on $\Gamma_h^{+}$, 
with $\eta_2 \in C^1(\Gamma_h^+)$, and 
the fact that $u\equiv  0$ on $\Gamma_h^-$, we can establish the following integral equality}
	\begin{equation}\label{eq:int387}
	0 =I_{1}^+ -  I_{12}^{-}+I_2+\lambda I_3,
	\end{equation} 
	where $I_1^+=-(I_{11}^+ + I_{12}^+)$,   $I_2$, $I_{12}^-$ and $I_3$ are defined in \eqref{eq:I11notaion},  \eqref{eq:Inotation}, \eqref{eq:I12notation} and \eqref{eq:int2} respectively.  For the term $I_{11}^+$, 
	it follows from \eqref{eq:asymptotic} that 
	\begin{align*} 
	I_{11}^+&=-\frac{2\beta(\theta_0) }{\zeta(\theta_0)  }u(\mathbf 0)	- \frac{1}{s} \cdot \frac{4\beta(\theta_0) }{\zeta(\theta_0)^3}c_1(\theta_0)-\frac{1}{s^2} \cdot \frac{48\beta(\theta_0) }{\zeta(\theta_0)^5}c_2(\theta_0) +\Oh(s^{-3}) 
	\end{align*}
	which can be further reduced to
	\begin{equation}\label{eq:388}
	I_{11}^+ =-\frac{1}{s^2} \cdot \frac{24\beta(\theta_0) }{\zeta(\theta_0)^5}\cdot \frac{\partial^2 u}{\partial x_1 x_2}\Big|_{\mathbf x=\mathbf 0} \sin 2\theta_0  +\Oh(s^{-3}) 
	\end{equation}
	by \eqref{eq:383}  and \eqref{eq:386}. 

\mm{
To estimate the term $I_{12}^+$, we 	recall that $ \eta_2$ has the expansion \eqref{eq:eta335}. 
We have $c_1(\theta_0 )=0$ from \eqref{eq:383} and \eqref{eq:386}. Then using 
$u( \mathbf 0)=c_1(\theta_0 )=0$, we deduce from \eqref{eq:asymptotic12} that
}
	\begin{align}\label{eq:768}
	I_{12}^+&= \Oh(s^{-2-\gamma }).
	\end{align}
	
\mm{
Next we estimate $I_{12}^-$. 	It follows from \eqref{eq:373int}, combining with \eqref{eq:383} and \eqref{eq:386}, 
that 
}
	\begin{align}\label{eq:390}
	I_{12}^-&= -\int_{\Gamma_h^-} u_0(s\mathbf  x) \frac{\partial u}{\partial \nu }  {\rm d} \sigma=\frac{12}{s^{2 }} \cdot       \frac{\partial^2 u}{\partial x_1 \partial x_2} \Big |_{\mathbf x=\mathbf 0}    \cdot  \frac{1}{ \zeta(0)^{4 } }-\Oh(s^{-3} ).
	\end{align}
Substituting \eqref{eq:388}-\eqref{eq:390} into \eqref{eq:int387}, we can get that
	\begin{align*}
	&\frac{1}{s^2} \cdot \frac{24\beta(\theta_0) }{\zeta(\theta_0)^5}\cdot \frac{\partial^2 u}{\partial x_1 x_2}\Big|_{\mathbf x=\mathbf 0} \sin 2\theta_0  - \frac{12}{s^{2 }} \cdot       \frac{\partial^2 u}{\partial x_1 \partial x_2} \Big |_{\mathbf x=\mathbf 0}    \cdot  \frac{1}{ \zeta(0)^{4 } }- \Oh\left(s^{-2-\gamma}\right) =- (I_2+\lambda I_3). 
	\end{align*}
	\mm{
	Multiplying $s^2$ on the both sides of the above equality, we deduce from \eqref{212} that 
	}
	\begin{equation*}
	\left(\frac{2\beta(\theta_0) }{\zeta(\theta_0)^5}\cdot \sin 2\theta_0  -           \frac{1}{ \zeta(0)^{4 } }\right) \cdot \frac{\partial^2 u}{\partial x_1 x_2}\Big|_{\mathbf x=\mathbf 0} =0
	\end{equation*}
	as $s\rightarrow \infty$. But we see from \eqref{eq:359eq} that 
	$$
	\frac{2\beta(\theta_0) }{\zeta(\theta_0)^5}\cdot \sin 2\theta_0  -           \frac{1}{ \zeta(0)^{4 } }=\bsi e^{-2 \bsi \theta_0} \sin 2 \theta_0-1=-\cos 2 \theta_0 e^{-2 \bsi \theta_0} \neq 0
	$$
	if $\theta_0\neq \pi/4$ and $\theta_0\neq 3\pi/4$. Hence we obtain that $$\frac{\partial^2u}{\partial x_1\partial x_2}\Big|_{\mathbf x=\mathbf 0}=0,$$ which completes the proof. 
\end{proof}

\subsection{Proof of Theorem \ref{th:47}}\label{subsec:th32}

\mm{
Using 
$ 
	u=0
$ 
on $\Gamma_h^{\pm}$, 
we have 
$$
\nabla u \Big|_{\mathbf x=\mathbf 0} \cdot (1,0)^\top =\nabla u \Big|_{\mathbf x=\mathbf 0} \cdot (\cos\theta_0 , \sin \theta_0 )^\top=0. 
$$
This implies 
}
\begin{equation}\label{eq:374gr}
	\nabla u \Big|_{\mathbf x=\mathbf 0}=0. 
\end{equation}
\mm{
Now we recall that $u$ has the expansion \eqref{eq:uex251}, 
then we can derive on $\Gamma_h^-$ by	using \eqref{eq:374gr} and polar coordinates
that  
}
$$
\sum_{ \alpha \in {\mathbb N}_0^2, \, |\alpha| \geq 2 \atop \alpha=(\alpha_1,\alpha_2) }\frac{(\partial^\alpha u)(\mathbf 0) }{\alpha ! }r^{|\alpha|} \cos^{\alpha_1 }( 0 )\sin^{\alpha_2} (0) \Big |_{ x\in \Gamma_h^- } \equiv 0,\quad 0\leq r\leq h\,,
$$
\mm{
from which it is not difficult to see that
\begin{equation}
	c_2(0)=\frac{1}{2} \frac{\partial^2 u}{\partial x_1^2}\Big|_{\mathbf x=\mathbf 0} \cdot \cos^2 0 =0\,.
	\end{equation}
	This can be used, along with \eqref{eq:2582nd}, to deduce that
	}
	\begin{equation}\label{eq:379}
		\frac{\partial^2 u}{\partial x_2^2}\Big|_{\mathbf x=\mathbf 0} =0\,.
	\end{equation}
	
\mm{By means of the fact $u=0$ on $\Gamma_h^{\pm}$ again, we have the integral identity}s
\begin{equation}\label{eq:int380}
	0 =-I_{12}^+ - I_{12}^{-}+I_2+\lambda I_3,
\end{equation}
where  $I_2$, $I_{12}^\pm$ and $I_3$ are defined in \eqref{eq:Inotation}, \eqref{eq:I12notation} and \eqref{eq:int2} respectively.  By Lemma \ref{prop710}, together  with \eqref{eq:374gr}-\eqref{eq:379}  it is easy to see that
\begin{equation}\label{eq:381}
	\begin{split}
		I_{12}^+&=\frac{12}{s^{2 }} \cdot      \frac{\partial^2 u}{\partial x_1 \partial x_2} \Big |_{\mathbf x=\mathbf 0} \sin  (\varphi_M + \theta_0 ) \cdot  \frac{1}{ \zeta(\theta_0)^{4 } } +\Oh(s^{-3}),\\
	I_{12}^-&=\frac{12}{s^{2 }} \cdot      \frac{\partial^2 u}{\partial x_1 \partial x_2} \Big |_{\mathbf x=\mathbf 0} \sin  (\varphi_m) \cdot  \frac{1}{ \zeta(0)^{4 } } +\Oh(s^{-3}),
	\end{split}
	\end{equation}
where $\varphi_m=-\pi/2$ and $\varphi_M=\theta_0+\pi/2$ are the arguments of the exterior unit normal vectors to $\Gamma_h^-$ and $\Gamma_h^{+}$ respectively. From \eqref{eq:beta}, we have
$$
\zeta(\theta_0)^4=e^{2\bsi \theta_0},\quad \zeta(0)^4=1. 
$$
Substituting \eqref{eq:381} into \eqref{eq:int380}, we derive that
\begin{equation}\label{eq:a4}
\frac{12}{s^2} \left (1-\cos 2 \theta_0 e^{-2\bsi \theta_0 }\right )  \frac{\partial^2 u}{\partial x_1 \partial x_2} \Big |_{\mathbf x=\mathbf 0} + I_2 +\lambda I_3 -\Oh(s^{-3})=0. 
\end{equation}
Now multiplying $s^2$ on the both sides of \eqref{eq:a4}, noting \eqref{212}, we have
$$
\left (1-\cos 2 \theta_0 e^{-2\bsi \theta_0 }\right )  \frac{\partial^2 u}{\partial x_1 \partial x_2} \Big |_{\mathbf x=\mathbf 0}=0
$$
from which we can deduce that
$$
 \frac{\partial^2 u}{\partial x_1 \partial x_2} \Big |_{\mathbf x=\mathbf 0}=0
$$
if $\theta_0\neq \pi/2$. 
This completes our proof.  \qed 


\subsection{Proof of Theorem \ref{Th:411}}\label{subsec:53}
Since $u$ satisfies the boundary condition $u\equiv 0$ on $\Gamma_h^-$, we know that
$$
u(\mathbf 0)=0,\quad \nabla u \Big |_{ \mathbf x= \mathbf 0} \cdot (1,0)^\top =0 \Longrightarrow \frac{\partial u }{\partial x_1 } \Big |_{\mathbf x=\mathbf 0}=0. 
$$
Recalling that $\varphi_M=\theta_0+\pi/2$ is the argument of the exterior unit normal vector of $\Gamma_h^+$ and $\partial_\nu u \equiv 0$ on  $\Gamma_h^+$, we easily see
$$
\frac{\partial u }{\partial x_1 } \Big |_{\mathbf x=\mathbf 0} (-\sin \theta_0)+ \frac{\partial u }{\partial x_2 } \Big |_{\mathbf x=\mathbf 0} \cos \theta_0=0. 
$$
Therefore if $\theta_0\neq \pi/2$, we know 
$$
\frac{\partial u }{\partial x_2 } \Big |_{\mathbf x=\mathbf 0} =0. 
$$
\mm{Furthermore, since $\Gamma_h^-$ is a nodal line, we know by Lemma \ref{pro:316} that}
\begin{equation*}
	\frac{\partial^2 u}{\partial x_1^2} \Big |_{\mathbf x=\mathbf 0}=0. 
\end{equation*}
Substituting  $u(\mathbf 0)=0$ into \eqref{eq:2582nd}, we have
\begin{equation*}
	\frac{ \partial^2 u } {\partial x_2^2} \Big |_{\mathbf x=\mathbf 0} =- \frac{ \partial^2 u } {\partial x_1^2} \Big |_{\mathbf x=\mathbf 0}=0.
\end{equation*}
\mm{Then using the fact that $\Gamma_h^+$ is a singular line of $u$, we derive from Lemma \ref{pro:312} that}
\begin{equation*}
		  \sin (\varphi_M +\theta_0)   \cdot \frac{ \partial^2 u } {\partial x_1 \partial x_2 } \Big|_{x=0} = \cos 2\theta_0  \cdot \frac{ \partial^2 u } {\partial x_1 \partial x_2 } \Big|_{\mathbf x=\mathbf 0} =0.
	\end{equation*}
Hence if $\theta_0\neq \pi/4$  and $\theta_0\neq 3\pi/4$, we can prove
$$
\frac{ \partial^2 u } {\partial x_1 \partial x_2 } \Big|_{\mathbf x=\mathbf 0}=0, 
$$
which means that the expansion \eqref{eq:uex251} of $u$ has at least nontrivial  homogeneous polynomial with the order of three, hence completes our proof. 
\qed 

\begin{remark}
	In the subsection \ref{subsec:53} above, we may also use the CGO solution as the test function to study the vanishing property of the second order partial derivatives of $u$ at the origin with respect to $\theta_0$, which can lead to the same conclusion.
\end{remark}

\section{Proofs of the theorems in Section \ref{sec:3} for general cases}\label{sec:pro2}

In this section, detailed proofs of the theorems for general vanishing order in Section \ref{sec:3} are presented,
\mm{
by using the spherical wave expansion of the Laplacian eigenfunction $u$ 
near the intersecting point between two line segments. 
}



 From \eqref{eq:uSex} and \eqref{eq:directional derivatitve}, we have the following lemma regarding the exterior normal derivative of $u$ on $\Gamma^\pm_h$ by using the spherical wave expansion \eqref{eq:uSex}  of $u$. 
\begin{lemma}\label{pro:61}
Under the polar coordinate, we have the following expansion of the normal derivative  of $u$ given by \eqref{eq:uSex} on $\Gamma_h^\pm $around the origin 
\begin{equation} \label{eq:41}
	\begin{split}
		\frac{\partial u}{\partial \nu} \Big|_{\Gamma_h^+} &= \frac{1}{r} \frac{\partial u}{\partial \theta}  \Big|_{\theta=\theta_0 } =\frac{1}{r}\sum_{n=0}^\infty\bsi n  \left( a_n e^{\bsi n \theta_0 }- b_n e^{-\bsi n \theta_0 }\right)J_n\left(\sqrt{\lambda}  r\right ),\\
		\frac{\partial u}{\partial \nu} \Big|_{\Gamma_h^-} &= -\frac{1}{r} \frac{\partial u}{\partial \theta}  \Big|_{\theta=0 } =-\frac{1}{r}  \sum_{n=0}^\infty\bsi n  \left( a_n - b_n  \right)J_n\left(\sqrt{\lambda}  r\right ). 
	\end{split}
\end{equation}
\end{lemma}


\mm{
Recalling the definition of the generalized singular line of the Laplacian eigenfunction $u$, 
and using the spherical wave expansion \eqref{eq:uSex} of $u$ and Lemma \ref{pro:61}, 
we can deduce some recursive equations for the undetermined coefficients $\{a_n, b_n\}$ 
in \eqref{eq:uSex}.  These recursive equations will be used in the proof of Theorem \ref{th:41}. 
}
\begin{lemma}\label{pro:62}
	Suppose that $\Gamma_h^\pm$ are two generalized singular lines of $u$ with the boundary parameters $\eta_1 \equiv C_1$ and $\eta_2  \equiv C_2$ defined on $\Gamma_h^-$ and $\Gamma_h^+$ respectively, where $C_1$ and $C_2$ are two constants. 
	\mm{
	If $\Gamma_h^\pm$ intersect with each other at the origin, 
	then the following recursive equations hold for the coefficients $\{a_n, b_n\}$ 
in \eqref{eq:uSex}, $n=1,2,\ldots$: }
	\begin{align}
	2 C_2 \left(a_{n-1} e^{2\bsi (n-1) \theta_0} + b_{n-1} \right) + \bsi   \sqrt{\lambda}e^{-\bsi \theta_0} \left(a_{n} e^{2 \bsi n \theta_0} - b_{n} \right) &=0,\label{eq:42} \\
	2 C_1 \left(a_{n-1}  + b_{n-1} \right) - \bsi   \sqrt{\lambda} \left(a_{n} - b_{n} \right) &=0.\label{eq:43}
\end{align}
\end{lemma}
\begin{proof}
Substituting \eqref{eq:uSex} and \eqref{eq:41} into  
	$$
	\frac{\partial u}{\partial \nu} +\eta_2 u=0 \mbox{ on } \Gamma_h^+,
	$$
	we obtain that
	\begin{equation}\label{eq:a6rc2}
	r C_2 \sum_{n=0}^\infty\left( a_n e^{\bsi n \theta_0 }+b_n e^{-\bsi n \theta_0 }\right)J_n\left(\sqrt{\lambda}  r\right ) + \sum_{n=0}^\infty\bsi n  \left( a_n e^{\bsi n \theta_0 }- b_n e^{-\bsi n \theta_0 }\right)J_n\left(\sqrt{\lambda}  r\right )=0. 
	\end{equation}
	Substituting \eqref{eq:jnt} into the equation above yields 
	\begin{align*}
	0&=	r C_2 \sum_{n=0}^\infty\left( a_n e^{\bsi n \theta_0 }+b_n e^{-\bsi n \theta_0 }\right)\frac{(\sqrt{\lambda} r)^n}{2^n n!} \left( 1+ \sum_{p=1}^\infty \frac{(-1)^pn!}{p!(n+p)!} \left(\frac{\sqrt{\lambda} r }{2} \right)^{2p}  \right  ) \\
	&\quad  + \sum_{n=0}^\infty\bsi n  \left( a_n e^{\bsi n \theta_0 }- b_n e^{-\bsi n \theta_0 }\right)\frac{(\sqrt{\lambda} r)^n}{2^n n!} \left( 1+ \sum_{p=1}^\infty \frac{(-1)^pn!}{p!(n+p)!} \left(\frac{\sqrt{\lambda} r }{2} \right)^{2p}  \right  ).
	\end{align*}
\mm{
Comparing the coefficient of $r^n$ on both sides, we can deduce for $n=1,2,\ldots,$ that 
}
$$
0=C_2 \left(a_{n-1} e^{\bsi (n-1) \theta_0} + b_{n-1} e^{-\bsi (n-1) \theta_0}\right) \cdot \frac{ \lambda^{(n-1)/2}}{ 2^{n-1} (n-1)!} + \bsi n \left(a_{n} e^{\bsi n \theta_0} - b_{n} e^{-\bsi n \theta_0}\right) \cdot \frac{ \lambda^{n/2}}{ 2^{n} n!} ,
$$
\mm{
which can be further simplified to get \eqref{eq:42}. Similarly, we can derive \eqref{eq:43} on $\Gamma_h^-$ 
by using \eqref{eq:uSex} and \eqref{eq:41}. 
}
\end{proof}

Similar to Lemma \ref{pro:62}, we can obtain the recursive equations  for $\{a_n,b_n\}$ by using \eqref{eq:uSex} and the  boundary conditions on $\Gamma_h^\pm$ for the nodal line and  the singular line in Propositions \ref{pro:63} and \ref{pro:64}, respectively. 

\begin{lemma}\label{pro:63}
	Suppose that $\Gamma_h^\pm$ are two nodal lines of $u$ which intersect  with each other at the origin, 
	then the following equations hold for the coefficients $\{a_n, b_n\}$ 
in \eqref{eq:uSex}, $n=0,1,\ldots$:
	\begin{align}
	a_n e^{\bsi n \theta_0 }+b_n e^{-\bsi n \theta_0 } &=0,\label{eq:47} \\
	 a_n +b_n  &=0.\label{eq:48}
\end{align}
\end{lemma}
\begin{proof}
	Substituting \eqref{eq:uSex} into 
	$
	u\equiv 0 \mbox{ on } \Gamma_h^\pm,
	$
	we can obtain that
	\begin{align*}
		0&= \sum_{n=0}^\infty\left( a_n e^{\bsi n \theta_0 }+b_n e^{-\bsi n \theta_0 }\right)J_n(\sqrt{\lambda}r),\\
		0&= \sum_{n=0}^\infty\left( a_n +b_n \right)J_n(\sqrt{\lambda}r).
	\end{align*}
\mm{	Then the desired results follow directly from Lemma \ref{lem:27}. }
\end{proof}

\begin{lemma}\label{pro:64}
	Suppose that $\Gamma_h^\pm$ are two singular lines of $u$ which intersect  with each other at the origin, 
	then the following equations hold for the coefficients $\{a_n, b_n\}$ 
in \eqref{eq:uSex}, $n=0,1,\ldots$:	
\begin{align}
	a_n e^{\bsi n \theta_0 }- b_n e^{-\bsi n \theta_0 } &=0,\label{eq:49} \\
	 a_n -b_n  &=0.\label{eq:410}
\end{align}
\end{lemma}
\begin{proof}
	Using \eqref{eq:41} and Lemma \ref{lem:27}, we can derive  \eqref{eq:49}  and \eqref{eq:410}. 
\end{proof}

In the next lemma, we clarify the relationship between the coefficients $a_n$, $b_n$ in  \eqref{eq:uSex} and  the vanishing order of $u$ at the origin. 

\begin{lemma}\label{pro:45}
	Suppose that $u$ has the spherical wave expansion \eqref{eq:uSex} around the origin, 
	\mm{
	and the coefficients $a_n$, $b_n$ in  \eqref{eq:uSex} satisfy
	}
	\begin{equation}\label{eq:411a0}
		a_0+b_0=0, a_n=b_n=0, \, n=1,2,\ldots,q, \,q\in\mathbb{N}_+,
	\end{equation}
	\mm{then the vanishing order $N$ of $u$ at the origin is given by
	$
	N= q+1. 
	$
	}
\end{lemma}
\begin{proof}
	Substituting \eqref{eq:411a0} into \eqref{eq:uSex} yields 
	$$
	u(x)=\sum_{n=q+1}^\infty\left( a_n e^{\bsi n \theta }+b_n e^{-\bsi n \theta }\right)J_n\left(\sqrt{\lambda}  r\right ), 
	$$
\mm{then we know from \eqref{eq:jnt} that the power of the lowest order 
in $J_n\left(\sqrt{\lambda}  r\right )$ with respect to $r$ is $n$. 
}
\end{proof}

In the rest of this section, we provide detailed proofs of theorems for general vanishing order in Section \ref{sec:3}. 

\subsection{Proof of Theorem~\ref{th:41}}\label{sub61}
	By Theorem \ref{th:311}, we have 
	$$
	u(\mathbf{ 0})=\frac{\partial u}{\partial x_1}\Big|_{\mathbf x=\mathbf  0}=\frac{\partial u}{\partial x_2}\Big|_{\mathbf x=\mathbf 0}=0. 
	$$
\mm{Then it follows from \eqref{eq:uSex} that}
	\begin{equation}\label{eq:a0b0}
	a_0+b_0=u(\mathbf 0)=0. 
	\end{equation}
	Substituting \eqref{eq:a0b0} into \eqref{eq:42} and \eqref{eq:43}, we can deduce that
$$
\begin{bmatrix}
	e^{2\bsi \theta_0} & -1\\
	1&-1
\end{bmatrix} \begin{bmatrix}
	a_1 \\ b_1
\end{bmatrix}=0\,, 
$$	
\mm{from which we can easily see $a_1=b_1=0$ if $\theta_0 \neq \pi $. 
Then using \eqref{eq:42} and \eqref{eq:43} again, we know that $a_2$ and $b_2$ satisfy 
$$
\begin{bmatrix}
	e^{4 \bsi \theta_0 } &-1\\
	1&-1
\end{bmatrix} \begin{bmatrix}
	a_2\\ b_2
\end{bmatrix}=0\,, 
$$
from which we see $a_2=b_2=0$ if $\theta_0\neq k \pi/2$ for $k=0,1$. 
Now we apply the mathematical induction, and assume that $a_{n-1}=b_{n-1}=0$. 
Then we can directly derive by virtue of \eqref{eq:th1} that 
\begin{equation}\label{eq:414}
	\left| \begin{matrix}
	e^{2\bsi n \theta_0} &-1\\
	1&-1
\end{matrix} \right| =1-e^{2\bsi n \theta_0} \neq 0 
\end{equation}
for $\theta_0 \neq \frac{m \pi }{n} \, (m=0,1,\ldots, n-1)$. This implies that $a_n=b_n=0$, 
and completes the proof of Theorem~\ref{th:41}.
}
\qed 

\subsection{Proof of Theorem~\ref{th:47}}
\mm{Since $u\equiv0$ on $\Gamma_h^\pm$, by Lemma \ref{pro:63} we know 
$$
\left| \begin{matrix}
e^{\bsi n \theta_0 } &  e^{-\bsi n \theta_0 }\\
1& 1
\end{matrix} \right| = e^{-\bsi n \theta_0 } ( e^{2\bsi n \theta_0 }-1) \neq  0
$$
if $\theta_0 \neq \frac{m \pi }{n} \, (m=0,1,\ldots, n-1)$. This readily 
implies that $a_n=b_n=0$, $n=1,2,\ldots$.  }
\qed

\subsection{Proof of Theorem~\ref{th:48}}
\mm{In combination with Lemma \ref{pro:64}, Theorem \ref{th:48} can be proved by following a completely similar argument to the one for Theorem \ref{th:41} in Subsection \ref{sub61}  
by formally taking $\eta_1\equiv 0$ and $\eta_2\equiv 0$. }

%

\subsection{Proof of Theorem~\ref{Th:49}}
Since $u\equiv0$ on $\Gamma_h^-$ and $\frac{\partial u}{\partial\nu}+\eta_2u=0$ on $\Gamma_h^+$, we have the following recursive equations for the coefficients $\{a_n,b_n\}$:
	\mm{
	\begin{equation}
		a_0+b_0=0; \q 
			 a_n +b_n  =0; \q 
			 2 C_2 (a_{n-1} e^{2\bsi (n-1) \theta_0} + b_{n-1} ) + \bsi   \sqrt{\lambda}e^{\bsi \theta_0} 
			 (a_{n} e^{2 \bsi n \theta_0} - b_{n} ) =0. \label{453}
	\end{equation}
	}
\mm{
For $n=1$, we easily see from \eqref{453} that 
\begin{align}\notag
a_1+b_1=0, \quad a_1e^{2\bsi\theta_0}-b_1=0, 
\end{align}
from which we readily derive $a_1=b_1=0$.}

\mm{
By mathematical induction, we assume $a_{n-1}=b_{n-1}=0$. Then there holds
\begin{align}\notag
	a_n+b_n=0, \quad 
	a_ne^{2\bsi n\theta_0}-b_n=0.
\end{align}
If $\theta_0 \neq \frac{(2m +1) \pi   }{2n}\, (m=0,1,\ldots, n-1)$, the coefficient matrix satisfies
$$
\left| \begin{matrix}
	 e^{2\bsi n \theta_0 } & -1\\
	 1& 1
\end{matrix} \right| = e^{2\bsi n \theta_0 }+1 \neq  0, 
$$
which readily shows that $a_n=b_n=0$, $n=1,2,\ldots$.  
}
%
\qed 

\section{Discussions about the condition $u(\mathbf{ 0})=0$}\label{sect:discussion}
\mm{We recall an essential condition $u(\mathbf{0})=0$ that was used 
in Theorems \ref{th:41},  \ref{th:48} and  \ref{th:410} in Section \ref{sec:3}. 
However, we may note that in other three theorems of the same section, 
the condition that $u(\mathbf{0})=0$ is always fulfilled because 
one of the two line segments is a nodal line there. 
In this section, by illustrating with several examples, we show that $u(\mathbf{0})=0$ can also be fulfilled in Theorems \ref{th:41},  \ref{th:48} and  \ref{th:410} if one imposes certain generic conditions on the boundary parameters $C_i$, the intersecting angle $\alpha \cdot \pi$ and the eigenvalue $\lambda$. 

It is stated in the introduction that one of the main motivations of our study in this work 
is the unique identifiability in inverse scattering problems. As we will see in the next section, 
we are able to develop a powerful mathematical strategy so that this condition is always fulfilled 
by making use of a linear combination of two eigenfunctions. 
}

\begin{proposition}\label{prop1}
	\mm{
	Let $u$ be a Laplacian eigenfunction to \eqref{eq:eig}, with its 
	Fourier series given by \eqref{eq:uSex}. Suppose that there are two generalized singular lines $\Gamma_h^+$ and $\Gamma_h^-$ from ${\mathcal M}^\lambda_{\Omega }$ such that \eqref{eq:angle} and \eqref{eq:angle2a} hold. 
	Assume that $\eta_1 \equiv C_1$ and $\eta_2 \equiv  C_2$ for two constants 
	$C_1 $ and $C_2$.  Then if $\alpha=1$ and $C_1\neq C_2$, the Laplacian eigenfunction $u$ 
	fulfills $u(\mathbf{ 0})=0$. If $\alpha \neq 1$, two coefficients $a_1$ and $b_1$ in \eqref{eq:uSex} 
	can be expressed explicitly by 
	\begin{align}\label{eq:72}
	a_1=\frac{1}{\sqrt{\lambda}\sin\theta_0}(C_1e^{-\bsi\theta_0}+C_2)u(\mathbf{ 0}), \q 
	b_1=\frac{1}{\sqrt{\lambda}\sin\theta_0}(C_1e^{\bsi\theta_0}+C_2)u(\mathbf{ 0}). 
	\end{align}
	}
\end{proposition}

\begin{proof}
	\mm{
	Recall the Laplacian eigenfunction $u$ has the spherical wave expansion \eqref{eq:uSex} in polar coordinates around the origin. 	
	Then we can obtain from Lemma \ref{pro:62}, and noting that $\alpha=1$ implies $\theta_0=\pi$, 
	the following equations 
	\begin{equation}\label{r1}
	\begin{cases}
	2C_2(a_0+b_0)-\bsi\sqrt{\lambda}  (a_1-b_1)  =0,\\
	2C_1(a_0+b_0)-\bsi \sqrt{\lambda} (a_1-b_1)  =0. 
	\end{cases}
	\end{equation}
	Noting that $a_0+b_0=u(\mathbf 0)$ and using the assumption $C_1 \neq C_2$, we can derive 
	from the above equations that $u(\mathbf 0 )=0$. And \eqref{eq:72} follows readily from \eqref{r1} 
	if $\alpha \neq 1$. 
	}
\end{proof}


\begin{proposition}\label{prop2}
	Let $u$ be a Laplacian eigenfunction to \eqref{eq:eig}. Suppose that there are two generalized singular lines $\Gamma_h^+$ and $\Gamma_h^-$ from ${\mathcal M}^\lambda_{\Omega }$ such that \eqref{eq:angle} and \eqref{eq:angle2a} hold. 
	\mm{
	Assume that $\eta_1 \equiv C_1$ and $\eta_2 \equiv  C_2$ for two constants 
	$C_1 $ and $C_2$.  If $\alpha =1/2$, there holds that $C_1C_2u(\mathbf{0})\equiv C_1C_2u(\mathbf{0})$. If $\alpha \neq\frac{1}{2}$, then $a_2$ and $b_2$ in \eqref{eq:uSex} can be expressed explicitly  by
	} 
	\begin{align}\label{eq:75}
	&a_2=\frac{2u(\mathbf{ 0})}{\lambda\sin2\theta_0\sin\theta_0}(C_1C_2+C_1C_2e^{-\bsi2\theta_0}+C_2^2\cos\theta_0+C_1^2\cos\theta_0e^{-\bsi2\theta_0}),\\&b_2=\frac{2u(\mathbf{ 0})}{\lambda\sin2\theta_0\sin\theta_0}(C_1C_2+C_1C_2e^{\bsi2\theta_0}+C_2^2\cos\theta_0+C_1^2\cos\theta_0e^{\bsi2\theta_0}). \notag
	\end{align}

\end{proposition}

\begin{proof}
Substituting \eqref{eq:72}	into \eqref{eq:42} and \eqref{eq:43} and taking $n=2$, we can obtain that
	\begin{equation}\label{r2}
	\begin{cases}
	2 C_2(a_1e^{\bsi\theta_0}+b_1e^{-\bsi\theta_0})+\bsi\sqrt{\lambda} (a_2e^{\bsi2\theta_0}-b_2e^{-\bsi2\theta_0}) =0,\\ 
	2 C_1(a_1+b_1) - \bsi  \sqrt{\lambda}(a_2-b_2) =0.
	\end{cases}
	\end{equation}
	which by the fact that $\theta_0=\frac{\pi}{2}$ can be further simplified as
	\mm{
	\begin{equation}\label{r22}
	-a_2+b_2=\dfrac{4\bsi C_2u(\mathbf{ 0})}{\lambda}C_1, \q 
	-(a_2-b_2)=\dfrac{4\bsi C_1u(\mathbf{ 0})}{\lambda}C_2
	\end{equation}
	}
 Since the system \eqref{r22} is consistent,  the right hand side of the equations belongs to the range of the coefficient matrix, therefore we can deduce the identity that 
 $C_1C_2u(\mathbf{0})\equiv C_1C_2u(\mathbf{0})$. 
 The result \eqref{eq:75} can be derived also directly if $\alpha \neq 1/2$.
\end{proof}


\begin{proposition}\label{prop3}
	Let $u$ be a Laplacian eigenfunction to \eqref{eq:eig}. Suppose that there are two generalized singular lines $\Gamma_h^+$ and $\Gamma_h^-$ from ${\mathcal M}^\lambda_{\Omega }$ such that \eqref{eq:angle} and \eqref{eq:angle2a} hold. 
   Assume that $\eta_1 \equiv C_1$ and $\eta_2 \equiv  C_2$ for two constants 
	$C_1 $ and $C_2$.  Then if $\alpha=\frac{1}{3}$, $C_1\neq C_2$ and
   \begin{equation}\label{pro:ass1}
   	 1 + \frac{4}{3\lambda }\left( C
	 _1^2 +C_1C_2+C_2^2\right) \neq 0, 
   \end{equation}
   the Laplacian eigenfunction $u$ fulfills $u(\mathbf{ 0})=0$. Furthermore, if 
   $\alpha \neq\frac{1}{3}$, then $a_3$ and $b_3$ in \eqref{eq:uSex} can be expressed explicitly by 
	\mm{
	\begin{equation}\label{eq:712}
	a_3=\frac{2}{\sqrt{\lambda}\sin3\theta_0}(B_1+B_2e^{-\bsi3\theta_0}),\q
	b_3=\frac{2}{\sqrt{\lambda}\sin3\theta_0}(B_1+B_2e^{\bsi3\theta_0}).  
	\end{equation}
	where $B_1$ and $B_2$ are defined by 
	\begin{equation}\label{B1}
	B_1=2C_2(a_2e^{\bsi2\theta_0}+b_2e^{-\bsi2\theta_0})-4C_2u(\mathbf{ 0})-\bsi (a_1e^{\bsi\theta_0}-b_1e^{-\bsi\theta_0})\sqrt{\lambda},
	\end{equation}
	\begin{equation}\label{B2}
	B_2=2C_1(a_2+b_2)-4C_1u(\mathbf{ 0})+\bsi (a_1-b_1)\sqrt{\lambda}, 
	\end{equation}
where $(a_1, b_1)$ and $(a_2, b_2)$ are given by \eqref{eq:72} and  \eqref{eq:75}, respectively. 
}
\end{proposition}

\begin{proof}
	Substituting  \eqref{eq:75} into \eqref{eq:42} and \eqref{eq:43} and taking $n=3$,  we can obtain that
	\mm{
	\begin{equation}\label{r3}
	\bsi(a_3e^{\bsi3\theta_0}-b_3e^{-\bsi3\theta_0})\sqrt{\lambda}=-B_1,\q
	-\bsi (a_3-b_3)\sqrt{\lambda}=-B_2
	\end{equation}
	}
	where $B_1$ and $B_2$ are defined by \eqref{B1} and \eqref{B2} respectively. 
Since $\alpha=\frac{1}{3}$, taking 	$\theta_0=\frac{\pi}{3}$ in \eqref{r3}, we have
	\mm{
	\begin{equation}\label{r33}
	-a_3+b_3=\bsi B_1\frac{1}{\sqrt{\lambda}}, \q 
	-(a_3-b_3)=\bsi B_2\frac{1}{\sqrt{\lambda}}
	\end{equation}
	}
	where 
	\begin{equation}
	B_1=2C_2(a_2e^{\bsi\frac{2\pi}{3}}+b_2e^{-\bsi\frac{2\pi}{3}})-4C_2u(\mathbf{ 0})-\bsi (a_1e^{\bsi\frac{\pi}{3}}-b_1e^{-\bsi\frac{\pi}{3}})\sqrt{\lambda},
	\end{equation}
	and $B_2$ is the same as \eqref{B2}. Since the system \eqref{r33} is consistent, the right hand side of the equations belongs to  the range of the coefficient matrix, which indicates that $B_1=B_2$, leading to 
	the relation 	
	\begin{equation}\label{eq12}
	\frac{4u(\mathbf{ 0})C_2}{3\lambda}\left(C_1C_2+C_1^2-\frac{1}{2}C_2^2\right)-\frac{C_2}{2}u(\mathbf{ 0})
	=\frac{4u(\mathbf{ 0})C_1}{3\lambda}\left(C_1C_2+C_2^2-\frac{1}{2}C_1^2\right)-\frac{C_1}{2}u(\mathbf{ 0}). 
	\end{equation}
	After straightforward calculations, \eqref{eq12} can be further reduced to
	$$
	 (C_1-C_2)\left[1 + \frac{4}{3\lambda }\left( C
	 _1^2 +C_1C_2+C_2^2\right) \right ] u(\mathbf{ 0})=0. 
	$$
	\mm{
	This implies $u(\mathbf{ 0})=0$ by noting that $C_1\neq C_2$ and using 
	\eqref{pro:ass1}. 
	Similarly we can deduce \eqref{eq:712} for $\alpha \neq 1/3$. 
	}
	\end{proof}


\begin{proposition}\label{prop4}
	Let $u$ be a Laplacian eigenfunction to \eqref{eq:eig}. Suppose that there are two generalized singular lines $\Gamma_h^+$ and $\Gamma_h^-$ from ${\mathcal M}^\lambda_{\Omega }$ such that \eqref{eq:angle} and \eqref{eq:angle2a} hold. 
    \mm{
    Assume that $\eta_1 \equiv C_1$ and $\eta_2 \equiv  C_2$ for two constants 
	$C_1 $ and $C_2$.  Then it holds for $\alpha=\frac{1}{4}$ that 
	\begin{equation}\label{pro:ass2}
    	u(\mathbf{ 0})(C_2^2-C_1^2)\equiv u(\mathbf{ 0})(C_2^2-C_1^2),
    \end{equation} 
     while for $\alpha \neq\frac{1}{4}$, the coefficients $a_4$ and $b_4$ in \eqref{eq:uSex} can be expressed explicitly by
     }
	\begin{align}\label{a4}
	&a_4=\frac{6}{\sqrt{\lambda}\sin4\theta_0}(D_1+D_2e^{-\bsi4\theta_0}),\\
	&b_4=\frac{6}{\sqrt{\lambda}\sin4\theta_0}(D_1+D_2e^{\bsi4\theta_0}).
	\end{align}
	where $D_1$ and $D_2$ are defined by 
	\begin{equation}\label{D1}
	D_1=2C_2(a_3e^{\bsi3\theta_0}+b_3e^{-\bsi3\theta_0})
	-\bsi2\sqrt{\lambda}(a_2e^{\bsi2\theta_0}-b_2e^{-\bsi2\theta_0})-6C_2(a_1e^{\bsi\theta_0}+b_1e^{-\bsi\theta_0}),
	\end{equation}
	\begin{equation}\label{D2}
	D_2=2C_1(a_3+b_3)
	+\bsi2\sqrt{\lambda}(a_2-b_2)-6C_1(a_1+b_1).
	\end{equation}
\mm{
Here $(a_1, b_1)$, $(a_2, b_2)$ and $(a_3, b_3)$ are given by \eqref{eq:72},  \eqref{eq:75} and 
 \eqref{eq:712}, respectively. 
 }
\end{proposition}

\begin{proof}
	Substituting \eqref{eq:712} into \eqref{eq:42} and \eqref{eq:43} and taking $n=4$, we can obtain that
	\mm{
	\begin{equation}\label{r4}
	\bsi(a_4e^{\bsi4\theta_0}-b_4e^{-\bsi4\theta_0})\sqrt{\lambda}=-D_1,\q
	-\bsi (a_4-b_4)\sqrt{\lambda}=-D_2
	\end{equation}
	}
	where $D_1$ and $D_2$ are defined by \eqref{D1} and \eqref{D2} respectively.   Since  $\alpha=\frac{1}{4}$, we can substitute  $\theta_0=\frac{\pi}{4}$ into \eqref{r4} to obtain that 
	\mm{
	\begin{equation}\label{r44}
	-a_4+b_4=\bsi D_1\frac{1}{\sqrt{\lambda}}, \q 
	-(a_4-b_4)=\bsi D_2\frac{1}{\sqrt{\lambda}}
	\end{equation}
	}
	where
	\begin{equation}
	D_1=2C_2(a_3e^{\bsi\frac{3\pi}{4}}+b_3e^{-\bsi\frac{3\pi}{4}})
	-\bsi2\sqrt{\lambda}(a_2e^{\bsi\frac{\pi}{2}}-b_2e^{-\bsi\frac{\pi}{2}})-6C_2(a_1e^{\bsi\frac{\pi}{4}}+b_1e^{-\bsi\frac{\pi}{4}}),
	\end{equation}
	and $D_2$ is the same as \eqref{D2}. Since the system \eqref{r44} is also consistent,  the right hand side of the equations belongs to the range of the coefficient matrix, which implies $D_1=D_2$, hence 
	\begin{align}\label{eq11}
	&\frac{2C_2}{3\sqrt{\lambda}}(\sqrt{2}B_2-B_1)+\frac{2u(\mathbf{ 0})}{3\sqrt{\lambda}}(\sqrt{2}C_1C_2+C_2^2)-\frac{C_2u(\mathbf{ 0})}{\sqrt{\lambda}}(\sqrt{2}C_1+C_2)
	\\
	&=\frac{2C_1}{3\sqrt{\lambda}}(\sqrt{2}B_1-B_2)+\frac{2u(\mathbf{ 0})}{3\sqrt{\lambda}}(\sqrt{2}C_1C_2+C_1^2)-\frac{C_1u(\mathbf{ 0})}{\sqrt{\lambda}}(\sqrt{2}C_2+C_1)\notag
	\end{align}
	where
	\begin{align}\label{B1+}
	\sqrt{2}B_2-B_1&=\frac{2\sqrt{2}u(\mathbf{ 0})C_1}{\lambda}(\sqrt{2}C_1C_2+C_2^2)-\frac{\sqrt{2}}{2}C_1u(\mathbf{ 0})\\
	&\quad -\frac{2u(\mathbf{ 0})C_2}{\lambda}\left(\sqrt{2}C_1C_2+C_1^2\right)+\frac{1}{2}C_2u(\mathbf{ 0}),\notag
	\end{align}
	\begin{align}\label{B2+}
	\sqrt{2}B_1-B_2
	&=\frac{2\sqrt{2}u(\mathbf{ 0})C_2}{\lambda}(\sqrt{2}C_1C_2+C_1^2)-\frac{\sqrt{2}}{2}C_2u(\mathbf{ 0})\\
	&\quad -\frac{2u(\mathbf{ 0})C_1}{\lambda}\left(\sqrt{2}C_1C_2+C_2^2\right)+\frac{1}{2}C_1u(\mathbf{ 0}).\notag
	\end{align}
	\mm{
	Substituting \eqref{B1+} and \eqref{B2+} into \eqref{eq11}, we deduce by 
	straightforward calculations that
	}
	\begin{equation}\notag
	u(\mathbf{ 0})(C_2^2-C_1^2)\equiv u(\mathbf{ 0})(C_2^2-C_1^2).
	\end{equation} 
	\eqref{a4} can be derived similarly for $\alpha \neq 1/4$, which completes the proof.
	\end{proof}


\mm{
We end this section with two important remarks.

\begin{remark}
	By tracing the proofs of Propositions \ref{prop1}--\ref{prop4} and repeating similar arguments, 
	we can find that under some mild assumptions on $C_1$, $C_2$, the intersecting angle  $\alpha \cdot \pi $ and the eigenvalue $\lambda$, the property $u(\mathbf{ 0}) = 0$ still holds for the rational intersecting 
	angle $\alpha \cdot \pi $ generically except for $\alpha=\pi/{2m}$, where $m=1,2,\cdots$. 
	The detailed arguments are rather tedious and technical, but straightforward. 
\end{remark}

\begin{remark}
	In Propositions \ref{prop1}-\ref{prop4}, we studied the property $u(\mathbf{ 0})=0$  
	for two intersected generalized singular lines only for some conditions on 
	$C_1$, $C_2$, the intersecting angle $\alpha \cdot \pi$ and $\lambda$. Other situations may be 
	analysed similarly, e.g., either $C_1=0$  or $C_2=0$. But as shown in Theorem \ref{th:28},  
	we can not guarantee $u(\mathbf 0 )=0$ by imposing some conditions on the intersecting angle
	between two intersecting singular lines. 
\end{remark} 
}

\mm{
\section{Unique identifiability for inverse scattering problems}\label{sec8}

In this section, we apply the spectral results we have established in the previous sections 
to study a fundamental mathematical topic, i.e., the unique identifiability,  
in a class of physically important inverse problems. 
These include the inverse obstacle problem and the inverse diffraction grating problem, 
which are concerned with imaging the shapes of some unknown or inaccessible objects 
from certain wave probing data in different physical settings. 
These inverse scattering problems may arise from a variety of important applications such as 
radar, sonar and medical imaging, as well as geophysical exploration and nondestructive testing. 
}

\subsection{Unique recovery for the inverse obstacle problem}

We first consider the inverse obstacle problem. Let $k=\omega/c\in\mathbb{R}_+$ be the wavenumber of
a time harmonic wave with $\omega\in\mathbb{R}_+$ and $c\in\mathbb{R}_+$, respectively, signifying the frequency and sound
speed. Let $\Omega\subset\R^2$ be a bounded domain with a
Lipschitz-boundary $\partial\Omega$ and a connected complement $\mathbb{R}^2\backslash\overline{\Omega}$.
Furthermore, let the incident field $u^i$ be a plane wave of the form
\begin{equation}\label{eq:pw}
u^i: =\ u^i(\mathbf x;k,\mathbf{d}) = e^{\mathrm{i}k \mathbf x\cdot \mathbf{d}},\quad x\in\mathbb{R}^2\,,
\end{equation}
where $\mathbf{d}\in \mathbb{S}^{1}$ denotes the incident direction of the impinging wave and $\mathbb{S}^{1}:=\{\mathbf x\in\mathbb{R}^2:|\mathbf x|=1\}$ is the unit circle in $\mathbb{R}^2$. Physically, $\Omega$ is an impenetrable obstacle that is unknown or inaccessible, and $u^i$ signifies the detecting wave field that is used for probing the obstacle. The presence of the obstacle interrupts the propagation of the incident wave, and generates the so-called scattered wave field $u^s$. Let $u:=u^i+u^s$ be the resulting total wave field, 
then the forward scattering problem can be described by the following Helmholtz system:
\begin{equation}\label{eq:scat1}
\begin{cases}
& \Delta u + k^2 u = 0\qquad\quad \mbox{in }\ \ \mathbb{R}^2\backslash\overline{\Omega},\medskip\\
& u =u^i+u^s\hspace*{1.56cm}\mbox{in }\ \ \mathbb{R}^2,\medskip\\
& \mathcal{B}(u)=0\hspace*{1.95cm}\mbox{on}\ \ \partial\Omega,\medskip\\
&\displaystyle{ \lim_{r\rightarrow\infty}r^{\frac{1}{2}}\left(\frac{\partial u^{s}}{\partial r}-\mathrm{i}ku^{s}\right) =\,0.}
\end{cases}
\end{equation}
The limiting equation above is known as the Sommerfeld radiation condition which holds uniformly in $\hat {\mathbf x}:={\mathbf x}/|{\mathbf x}|\in \mathbb{S}^{1}$ and characterizes the outgoing nature of the scattered wave field $u^s$. The boundary operator $\mathcal{B}$ could be Dirichlet type, $\mathcal{B}(u)=u$; or Neumann type, $\mathcal{B}(u)=\partial_\nu u$; or Robin type, $\mathcal{B}(u)=\partial_\nu u+\eta u$, corresponding to that $\Omega$ is a sound-soft, sound-hard or impedance obstacle, respectively. 
Here $\nu$ denotes the exterior unit normal vector to $\partial\Omega$ and $\eta\in L^\infty(\partial\Omega)$ signifies a boundary impedance parameter. It is required that $\Re\eta\geq 0$ and $\Im \eta\geq 0$. 
In what follows, we formally take $u=0$ on $\partial\Omega$ as $\partial_\nu u+\eta u=0$ on $\partial\Omega$ with $\eta=+\infty$. In doing so, we can unify all three boundary conditions as the generalized impedance boundary condition:
\begin{equation}\label{eq:gbc}
\mathcal{B}(u)=\partial_\nu u+\eta u=0\quad\mbox{on}\ \ \partial\Omega,
\end{equation}
where $\eta$ could be $\infty$, corresponding to a sound-soft obstacle. The forward scattering problem \eqref{eq:scat1} is well understood \cite{CK, Mclean} and there exists a unique solution $u\in H^1_{loc}(\mathbb{R}^2\backslash\overline{\Omega})$ that admits the following asymptotic expansion:
\begin{equation}\label{eq:far}
u^s(\mathbf x,\mathbf d,k)
=\frac{e^{\mathrm{i}kr}}{r^{{1/2}}} u_{\infty}(\hat{\mathbf x};k,\mathbf{d})+\mathcal{O}\left(\frac{1}{r^{3/2}}\right)\quad\mbox{as }\,r\rightarrow\infty
\end{equation}
which holds uniformly with respect to all directions $\hat {\mathbf x}:=\mathbf  x/|\mathbf  x|\in \mathbb{S}^{1}$.
The complex valued function $u_\infty$ in \eqref{eq:far} defined over the unit sphere $\mathbb{S}^{1}$
is known as the far-field pattern with $\hat{\mathbf x}\in \mathbb{S}^{1}$ signifying the observation direction. 
{\it The inverse obstacle scattering problem} is to recover $\Omega$ by using the knowledge of the far-field pattern $u_{\infty}(\hat{\mathbf x},\mathbf{d},k)$. By introducing an operator $\mathcal{F}$ which sends the obstacle to the corresponding far-field pattern through the Helmholtz system \eqref{eq:scat1}, the inverse obstacle problem can be formulated as the following abstract operator equation:
\begin{equation}\label{eq:iop}
\mathcal{F}(\Omega,\eta)=u_\infty(\hat {\mathbf x};k, \mathbf{d})\,,
\end{equation}
where $\mathcal{F}$ is defined by the forward obstacle scattering system, and is nonlinear. 
That is, one intends to determine  $(\Omega,\eta)$ from the knowledge of $u_\infty(\hat {\mathbf x};k, \mathbf{d})$.  

 A primary issue for the inverse obstacle problem \eqref{eq:iop} is the unique identifiability, which is concerned with the sufficient conditions such that the correspondence between $\Omega$ and $u_\infty$ is one-to-one. There is a widespread belief that one can establish  
uniqueness for \eqref{eq:iop} by a single or at most finitely many far-field patterns. We remark that 
by a single far-field pattern we mean that $u_\infty(\hat {\mathbf x}, k, \mathbf{d})$ is collected for all $\hat {\mathbf x}\in\mathbb{S}^1$, but is associated with a fixed incident 
$e^{\mathrm{i}k\mathbf x\cdot \mathbf{d}}$. Phrased in the geometric term, it states that the analytic function $u_\infty$ on the unit sphere associated with at most finitely many $k$ and $d$ can supply a global parameterization of a generic domain $\Omega$. This problem is known as the {\it Schiffer problem} in the inverse scattering community. It is named after M. Schiffer for his pioneering contribution around 1960 which is actually appeared as a private communication in the monograph by Lax and Phillips \cite{LP}. There is a long and colourful history on the study of the Schiffer problem, and we refer to a recent survey paper by Colton and Kress \cite{CK18} which contains an excellent account of the historical development of this problem. 

Recent progress on the Schiffer problem is made on general polyhedral obstacles in $\mathbb{R}^n$, $n\geq 2$. Uniqueness and stability results by using a finite number of far-field patterns can be found in \cite{AR,CY,LPRX,LRX,Liu-Zou,Liu-Zou3}. The major idea is to make use of the reflection principle for the Laplacian eigenfunction to propagate the so-called Dirichlet or Neumann hyperplanes. In the two-dimensional case, the Dirichlet and Neumann hyperplanes are actually the nodal and singular lines introduced in the present paper. In \cite{Liu-Zou3}, 
\mm{two of the authors of the present paper made an effort to answer the unique determination issue 
for impedance-type obstacles but gave only a partial solution to this fundamental problem.}
%
\mm{In this section, we develop a completely new approach 
that is able to provide a solution to this inverse obstacle problem in two dimensions, 
and the approach is uniform to sound-soft, sound-hard and impedance type obstacles.}
\mm{
The new approach is completely local, and enables us to show in a rather 
general scenario that 
one can determine the convex hull of an impedance obstacle as well as its surface impedance on the boundary of the convex hull by at most two far-field patterns.}

Consider an obstacle $\Omega$ associated with the generalized impedance boundary condition \eqref{eq:gbc}. It is called an admissible polygonal obstacle if $\Omega\subset\mathbb{R}^2$ is an open polygon, and on each edge of $\partial\Omega$, $\eta$ is either a constant (possibly zero) or $\infty$. That is, each edge $\mathcal{K}$ of an admissible polygonal obstacle is either sound-soft ($\eta\equiv \infty$ on $\mathcal{K}$), or sound-hard ($\eta\equiv 0$ on $\mathcal{K}$), or impedance-type ($\eta$ is a constant on $\mathcal{K}$). It is emphasized that $\eta$ may take different values on different edges of $\partial\Omega$. We write $(\Omega, \eta)$ to signify an admissible polygonal obstacle. 

\begin{definition}\label{def:r1}
Let $(\Omega, \eta)$ be an admissible polygonal obstacle. If  all the angles of its corners are irrational, then it is said to be an \emph{irrational obstacle}. If there is a corner angle of $\Omega$ is rational, then it is called a \emph{rational obstacle}. The smallest degree of the rational corner angles of $\Omega$ is referred to as the \emph{rational degree} of $\Omega$. 
\end{definition}
	It is easy to see that for a rational polygonal obstacle $\Omega$ in Definition \ref{def:r1}, the rational degree of $\Omega$ is at least $2$. 

\begin{definition}\label{def:r2}
$\Omega$ is said to be an admissible complex polygonal obstacle if it consists of finitely many admissible polygonal obstacles. That is,
\begin{equation}\label{eq:r2a}
(\Omega, \eta)=\bigcup_{j=1}^l (\Omega_j, \eta_j),
\end{equation}
where $l\in\mathbb{N}$ and each $(\Omega_j, \eta_j)$ is an admissible polygonal obstacle. Here, we define
\begin{equation}\label{eq:r2b}
\eta=\sum_{j=1}^l \eta_j\chi_{\partial\Omega_j}. 
\end{equation}
Moreover, $\Omega$ is said to be irrational if all of its component polygonal obstacles are irrational, otherwise it is said to be rational. For the latter case, the smallest degree among all the degrees of its rational components is defined to be the degree of the complex obstacle $\Omega$.  
\end{definition}

Next, we first consider the determination of an admissible complex irrational polygonal obstacle by at most two far-field patterns. We have the following local uniqueness result.


\red{
\begin{theorem}\label{thm:uniqueness1}
Let $(\Omega, \eta)$ and $(\widetilde\Omega, \widetilde\eta)$ be two admissible complex irrational obstacles. Let $k\in\mathbb{R}_+$ be fixed and $\mathbf{d}_\ell$, $\ell=1, 2$ be two distinct incident directions from $\mathbb{S}^1$.  Let $\mathbf{G}$ denote the unbounded connected component of $\mathbb{R}^2\backslash\overline{(\Omega\cup\widetilde\Omega)}$. Let $u_\infty$ and $\widetilde{u}_\infty$ be, respectively, the far-field patterns associated with $(\Omega, \eta)$ and $(\widetilde\Omega, \widetilde\eta)$. If 
\begin{equation}\label{eq:cond1}
u_\infty(\hat {\mathbf x}, \mathbf{d}_\ell )=\widetilde u_\infty(\hat {\mathbf x}, \mathbf{d}_\ell), \ \ \hat{\mathbf x}\in\mathbb{S}^1, \ell=1, 2,
\end{equation}
then one has that
$$
\left(\partial \Omega \backslash \partial \overline{ \widetilde{\Omega }} \right  )\cup \left(\partial \widetilde{\Omega } \backslash \partial \overline{ \Omega } \right)
$$
cannot have a corner on  $\partial \mathbf{G}$.
\end{theorem}

}


\begin{proof}
We prove the theorem by contradiction. \red{Assume \eqref{eq:cond1}  holds but $
\left(\partial \Omega \backslash \partial \overline{ \widetilde{\Omega }} \right  )\cup \left(\partial \widetilde{\Omega } \backslash \partial \overline{ \Omega } \right)
$ has a corner $\mathbf x_c$ on $\partial \mathbf{G}$.  Clearly, $\mathbf x_c$ is either a vertex of $\Omega$ or a vertex of $\widetilde\Omega$. Without loss of generality, we assume that $\mathbf x_c$ is a vertex
of $\widetilde\Omega$. Moreover, we see that $\mathbf{x}_c$ lies outside $\Omega$. Let $h\in\mathbb{R}_+$ be sufficiently small such that $B_h(x_c)\Subset\mathbb{R}^2\backslash\overline \Omega $. Moreover, since $x_c$ is a vertex of $\widetilde\Omega$, we can assume that 
\begin{equation}\label{eq:aa2}
B_h(\mathbf x_c)\cap \partial\widetilde\Omega=\Gamma_h^\pm,
\end{equation}
where $\Gamma_h^\pm$ are the two line segments lying on the two edges of $\widetilde\Omega$ that intersect at $\mathbf x_c$.  

Recall that $\mathbf{G}$ denotes the unbounded connected component of $\mathbb{R}^2\backslash\overline{(\Omega\cup\widetilde\Omega)}$. By \eqref{eq:cond1} and the Rellich theorem (cf. \cite{CK}), we know that
\begin{equation}\label{eq:aa3}
u(\mathbf x; k, \mathbf{d}_\ell)=\widetilde{u}(\mathbf x; k, \mathbf{d}_\ell),\quad x\in\mathbf{G},\ \ell=1, 2. 
\end{equation}
It is clear that $\Gamma_h^\pm\subset\partial\mathbf{G}$. Hence, by using \eqref{eq:aa3} as well as the generalized boundary condition \eqref{eq:gbc} on $\partial\widetilde\Omega$, we readily have 
\begin{equation}\label{eq:aa4}
\partial_\nu u+\widetilde\eta u=\partial_\nu \widetilde u+\widetilde\eta\widetilde u=0\quad\mbox{on}\ \ \Gamma_h^\pm. 
\end{equation}
It is also noted that in $B_h(\mathbf x_c)$, $-\Delta u=k^2 u$. Next, we consider two separate cases. 
}


\medskip

\noindent {\bf Case 1.}~Suppose that either $u(\mathbf x_c; k, \mathbf{d}_1)$ or $u(\mathbf x_c; k, \mathbf{d}_2)$ is zero. Without loss of generality, we assume that $u(\mathbf x_c; k, \mathbf{d}_1)=0$. By the assumption of the theorem that $\widetilde\Omega$ is an irrational obstacle, we see that $\Gamma_h^+$ and $\Gamma_h^-$ intersect with an irrational angle. Hence, by our results in Sections~\ref{sec:irra} and \ref{sec:irrational gene}, one immediately has that
\begin{equation}\label{eq:aa5}
u(\mathbf x; k, \mathbf{d}_1)=0\quad\mbox{in}\ \ B_h(\mathbf x_c),
\end{equation}
which in turn yields by the analytic continuation that 
\begin{equation}\label{eq:aa5}
u(\mathbf x; k, \mathbf{d}_1)=0\quad\mbox{in}\ \ \mathbb{R}^2\backslash\overline{\Omega}. 
\end{equation}
In particular, one has from \eqref{eq:aa5} that
\begin{equation}\label{eq:aa6}
\lim_{|\mathbf x|\rightarrow\infty} \left|u(\mathbf x; k, \mathbf{d}_1)\right|=0. 
\end{equation}
But this contradicts to the fact that follows from \eqref{eq:far}:
\begin{equation}\label{eq:aa6}
\lim_{|\mathbf x|\rightarrow\infty} \left|u(\mathbf  x; k, \mathbf{d}_1)\right|=\lim_{|\mathbf x|\rightarrow\infty} \left|e^{\mathrm{i}k\mathbf x\cdot \mathbf{d}_1}+u^s(\mathbf x; k, \mathbf{d}_1)\right|=1. 
\end{equation}

\medskip

\noindent {\bf Case 2.}~ Suppose that both $u(\mathbf x_c; k, \mathbf{d}_1)\neq 0$ and $u(\mathbf x_c; k, \mathbf{d}_2)\neq 0$. Set
\begin{equation}\label{eq:bb1}
\alpha_1=u(\mathbf x_c; k, \mathbf{d}_2)\quad\mbox{and}\quad \alpha_2=-u(\mathbf x_c; k, \mathbf{d}_1),
\end{equation}
and
\begin{equation}\label{eq:bb2}
v(\mathbf x)=\alpha_1 u(\mathbf x; k, \mathbf{d}_1)+\alpha_2 u(\mathbf x; k, \mathbf{d}_2),\ \ \ x\in B_h(\mathbf x_c). 
\end{equation}
Clearly, there hold
\mm{
\begin{equation}\label{eq:bb3}
-\Delta v=k^2 v\quad\mbox{in}\ \ B_h(\mathbf x_c); \q 
\partial_\nu v+\widetilde\eta v=0\quad\mbox{on}\ \ \Gamma_h^\pm. 
\end{equation}
}
Moreover, by the choice of $\alpha_1, \alpha_2$ in \eqref{eq:bb1}, one obviously has that $v(\mathbf x_c)=0$. Hence, by our results in Sections~\ref{sec:irra} and \ref{sec:irrational gene}, one immediately has that
\begin{equation}\label{eq:bb5}
v=0\quad\mbox{in}\ \ B_h(\mathbf x_c),
\end{equation}
which in turn yields by the analytic continuation that 
\begin{equation}\label{eq:bb6a}
\alpha_1 u(\mathbf x; k, \mathbf{d}_1)+\alpha_2 u(\mathbf x; k, \mathbf{d}_2)=0\quad\mbox{in}\ \ \mathbb{R}^2\backslash\overline{\Omega}. 
\end{equation}
However, since $\mathbf{d}_1$ and $\mathbf{d}_2$ are distinct, we know from \cite[Chapter 5]{CK} that $u(x; k, \mathbf{d}_1)$ and $u(x; k, \mathbf{d}_2)$ are linearly independent in $\mathbb{R}^2\backslash\overline{\Omega}$.  Therefore, one has from \eqref{eq:bb6a} that $\alpha_1=\alpha_2=0$, which contracts to the assumption at the beginning that both $\alpha_1$ and $\alpha_2$ are nonzero. 
%
\end{proof}

\red{
 It is recalled that the convex hull of $\Omega$, denoted by $\mathcal{CH}(\Omega)$, is the smallest convex set that contains $\Omega$. As a direct consequence of Theorem \ref{thm:uniqueness1}, we next show that the convex hull of a complex irrational obstacle can be uniquely determined by at most two far-field measurements. Furthermore, the boundary impedance  parameter $\eta$ can be partially identified as well.  In fact we have}
\begin{corollary}\label{co:84}
	Let $(\Omega, \eta)$ and $(\widetilde\Omega, \widetilde\eta)$ be two admissible complex irrational obstacles. Let $k\in\mathbb{R}_+$ be fixed and $\mathbf{d}_\ell$, $\ell=1, 2$ be two distinct incident directions from $\mathbb{S}^1$.     Let $\mathbf{G}$ denote the unbounded connected component of $\mathbb{R}^2\backslash\overline{(\Omega\cup\widetilde\Omega)}$.Let $u_\infty$ and $\widetilde{u}_\infty$ be, respectively, the far-field patterns associated with $(\Omega, \eta)$ and $(\widetilde\Omega, \widetilde\eta)$. If 
\begin{equation}\label{eq:cond1new}
u_\infty(\hat {\mathbf x}, \mathbf{d}_\ell )=\widetilde u_\infty(\hat {\mathbf x}, \mathbf{d}_\ell), \ \ \hat{\mathbf x}\in\mathbb{S}^1, \ell=1, 2,
\end{equation}
then one has that
\begin{equation}\label{eq:cond2}
\mathcal{CH}(\Omega)=\mathcal{CH}(\widetilde\Omega):=\Sigma,
\end{equation}
and
\begin{equation}\label{eq:cond3}
\eta=\widetilde\eta\ \ \mbox{on}\ \ \partial\Omega\cap\partial\widetilde\Omega\cap \partial\Sigma. 
\end{equation}	
\end{corollary}
\begin{proof}
	\mm{From Theorem \ref{thm:uniqueness1}, we can immediately conclude \eqref{eq:cond2}. Next we prove \eqref{eq:cond3}.} 
Let $\mathcal{E}\subset \partial\Omega\cap\partial\widetilde\Omega\cap \partial\Sigma$ be an open subset such that $\eta\neq \widetilde\eta$ on $\mathcal{E}$. By taking a smaller subset of $\mathcal{E}$ if necessary, we can assume that $\eta$ (respectively, $\widetilde\eta$) is either a fixed constant or $\infty$ on $\mathcal{E}$. Clearly, one has $u=\widetilde u$ in $\mathbb{R}^2\backslash\overline{\Sigma}$. Hence, there hold that 
\begin{equation}\label{eq:bb6}
\partial_\nu u+\eta u=0, \ \ \partial_\nu\widetilde u+\widetilde\eta \widetilde u=0,\ \ u=\widetilde u, \ \ \partial_\nu u=\partial_\nu\widetilde u\quad\mbox{on}\ \ \mathcal{E}. 
\end{equation}
By direct verification, one can show that 
\begin{equation}\label{eq:bb7}
u=\partial_\nu u=0\quad\mbox{on}\ \ \mathcal{E},
\end{equation}
which in turn yields by the Homogren's uniqueness result (cf. \cite{Liu-Zou}) that $u=0$ in $\mathbb{R}^2\backslash\Omega$. Hence, we arrive at the same contradiction as that in \eqref{eq:aa6}, which implies \eqref{eq:cond3}.
\end{proof}

\begin{remark}\label{rem:uniq1}
Let $\mathcal{V}(\Omega)$ and $\mathcal{V}(\mathcal{CH}(\Omega))$ denote, respectively, the sets of vertices of $\Omega$ and $\mathcal{CH}(\Omega)$. It is known that $\mathcal{V}(\mathcal{CH}(\Omega))\subset\mathcal{V}(\Omega)$. \red{Corollary~\ref{co:84}} states that if the corner angle of the polygon $\Omega$ at any vertex in $\mathcal{V}(\Omega)$ is irrational, then $\mathcal{CH}(\Omega)$ can be uniquely determined by two far-field patterns. Indeed, \red{from the proof of Theorem}~\ref{thm:uniqueness1}, we see that this requirement can be relaxed to that the corner angle of the polygon $\Omega$ at any vertex in $\mathcal{V}(\mathcal{CH}(\Omega))$ is irrational. 
\end{remark}

%
%

We proceed now to consider the unique determination of rational obstacles. Let $\Omega$ be a polygon in $\mathbb{R}^2$ and $\mathbf x_c$ be a vertex of $\Omega$. In what follows, we define
\begin{equation}\label{eq:region1}
\Omega_r(\mathbf x_c)=B_r(\mathbf x_c)\cap \mathbb{R}^2\backslash\overline{\Omega},\ \ r\in\mathbb{R}_+. 
\end{equation}
For a function $f\in L_{loc}^2(\mathbb{R}^2\backslash\overline{\Omega})$, we define 
\begin{equation}\label{eq:l1}
\mathcal{L}(f)(\mathbf x_c):=\lim_{r\rightarrow+0}\frac{1}{|\Omega_r(\mathbf x_c)|}\int_{\Omega_r(\mathbf x_c)} f(\mathbf x)\ {\rm d} \mathbf x
\end{equation}
if the limit exists. It is easy to see that if $f(\mathbf x)$ is continuous in $\overline{\Omega_{\tau_0}(\mathbf x_c)}$ for a sufficiently small $\tau_0\in\mathbb{R}_+$, then $\mathcal{L}(f)(\mathbf x_c)=f(\mathbf x_c)$. 

\begin{theorem}\label{thm:uniqueness2}
Let $(\Omega, \eta)$ be an admissible complex rational obstacle of degree $p\geq 3$. Let $k\in\mathbb{R}_+$ be fixed and $\mathbf{d}_\ell$, $\ell=1, 2$ be two distinct incident directions from $\mathbb{S}^1$. Set $u_\ell(\mathbf  x)=u(\mathbf x; k, \mathbf{d}_\ell)$ to be the total wave fields associated with $(\Omega, \eta)$ and $e^{\mathrm{i}k\mathbf x\cdot\mathbf{d}_\ell}$, $\ell=1, 2$, respectively. \red{ Recall that $\mathbf{G}$ denotes the unbounded connected component of $\mathbb{R}^2\backslash\overline{(\Omega\cup\widetilde\Omega)}$.} If the following condition is fulfilled, 
\begin{equation}\label{eq:cc1}
\mathcal{L}\left( u_2\cdot\nabla u_1-u_1\cdot\nabla u_2\right)(\mathbf x_c)\neq 0,
\end{equation}
 where $\mathbf x_c$ is any vertex of $\Omega$, \red{then one has that
 $$
\left(\partial \Omega \backslash \partial \overline{ \widetilde{\Omega }} \right  )\cup \left(\partial \widetilde{\Omega } \backslash \partial \overline{ \Omega } \right)
$$
cannot have a corner on  $\partial \mathbf{G}$.}

\end{theorem}

\begin{proof}
We prove the theorem by contradiction. Assume that there exists an admissible complex rational obstacle of degree $p\geq 3$, $(\widetilde\Omega,\widetilde\eta)$, such that \eqref{eq:cond1} holds but $\left(\partial \Omega \backslash \partial \overline{ \widetilde{\Omega }} \right  )\cup \left(\partial \widetilde{\Omega } \backslash \partial \overline{ \Omega } \right)
$ has  a corner on  $\partial \mathbf{G}$. In what follows, we adopt the same notation as those introduced in the proof of Theorem~\ref{thm:uniqueness1}. Note that the total wave fields $\widetilde u_\ell$, $\ell=1, 2$, associated with $(\widetilde\Omega,\widetilde\eta)$, are also assumed to fulfil the condition \eqref{eq:cc1}. 

By following a similar argument to the proof of Theorem~\ref{thm:uniqueness1}, one can show that there exist two line segments $\Gamma_h^\pm$ in $\mathbb{R}^2\backslash\overline{\Omega}$ such that $\partial_\nu u+\widetilde\eta u=0$ on $\Gamma_h^\pm$, and $\Gamma_h^+$ and $\Gamma_h^-$ intersect at a point $x_c$ which is a vertex of $\widetilde\Omega$. Using the fact that $u=\widetilde u$ near $\mathbf x_c$ and the condition \eqref{eq:cc1} on $(\widetilde\Omega,\widetilde\eta)$, we actually have
\begin{equation}\label{eq:cc2}
u(\mathbf x_c; \mathbf{d}_2)\cdot\nabla u(\mathbf x_c; \mathbf{d}_1)-u(\mathbf x_c; \mathbf{d}_1)\cdot\nabla u(\mathbf x_c; \mathbf{d}_2)\neq 0. 
\end{equation}
Clearly, \eqref{eq:cc2} implies that $\alpha_1:=u(\mathbf x_c; \mathbf{d}_2)$ and $\alpha_2=-u(\mathbf x_c; \mathbf{d}_1)$ cannot be identically zero. Set $v$ to be the one introduced in \eqref{eq:bb2}, then 
it clearly holds
\begin{equation}\label{eq:cc3}
v(\mathbf x_c)=0\quad\mbox{and}\quad \nabla v(\mathbf x_c)\neq 0. 
\end{equation}

Since $\widetilde\Omega$ is rational of degree $p\geq 3$, we know that the $\Gamma_h^+$ and $\Gamma_h^-$ intersect either at an irrational angle or a rational angle of degree $p\geq 3$. In either case, by our results in Sections~\ref{sec:irra}, \ref{sec:3} and \ref{sec:irrational gene}, we see that $v$ is vanishing at least to second order at $\mathbf x_c$. Hence, there holds $\nabla v(\mathbf x_c)=0$, which is a contradiction to \eqref{eq:cc3}. 
%
\end{proof}

\red{ Similar to Corollary \ref{co:84},  as a direct consequence of Theorem \ref{thm:uniqueness2}, under the condition \eqref{eq:cc1},  we next show that the convex hull of a complex rational obstacle of degree $p\geq 3$ can be uniquely determined by at most two far-field measurements.  Indeed  we have}
\begin{corollary}\label{co:87}
Let $(\Omega, \eta)$ be an admissible complex rational obstacle of degree $p\geq 3$. Let $k\in\mathbb{R}_+$ be fixed and $\mathbf{d}_\ell$, $\ell=1, 2$ be two distinct incident directions from $\mathbb{S}^1$. Set $u_\ell(\mathbf  x)=u(\mathbf x; k, \mathbf{d}_\ell)$ to be the total wave fields associated with $(\Omega, \eta)$ and $e^{\mathrm{i}k\mathbf x\cdot\mathbf{d}_\ell}$, $\ell=1, 2$, respectively.  If \eqref{eq:cc1}  is fulfilled, then	 $\mathcal{CH}(\Omega)$ is uniquely determined by $u_\infty(\hat{\mathbf x}, \mathbf{d}_\ell )$, $\ell=1,2$. Similar to Corollary \ref{co:84}, the boundary impedance  parameter $\eta$ can be partially identified as well.
\end{corollary}

\begin{remark}\label{remark:86}

As mentioned earlier that a general rational obstacle is at least of order $2$. By Remark~\ref{re:38}, we can easily extend the proof of Theorem~\ref{thm:uniqueness2} to cover the general case that $p=2$. However, as discussed in Remark \ref{re:38}, we need to exclude the case that $\eta\equiv\infty$ and $\eta$ is a finite number (possibly being zero) respectively on the two intersecting line segments $\Gamma_h^\pm$ (as appeared in the proof of Theorem~\ref{thm:uniqueness2}). 
\end{remark}

\begin{remark}\label{rem:uniq2}
Similar to Remark~\ref{rem:uniq1}, the condition \eqref{eq:cc1} can be relaxed to hold only at any vertex in $\mathcal{V}(\mathcal{CH}(\Omega))$ in Theorem \ref{thm:uniqueness2} and \red{Corollary \ref{co:87}}. Furthermore, since in the proof of Theorem~\ref{thm:uniqueness2}, we only make use of the vanishing up to the second order. By our results in Section~\ref{sec:pro1}, we know that Theorem~\ref{thm:uniqueness2} actually holds for a more general case where the surface impedance $\eta$ can be a $C^1$ function. 
\end{remark}

It would be interesting to investigate the sufficient conditions for \eqref{eq:cc1} to hold. If the obstacle $\Omega$ is sufficiently small compared to the wavelength, namely $k\cdot\mathrm{diam}(\Omega)\ll 1$, 
then from a physical viewpoint, the scattered wave field due to the obstacle is of a much smaller magnitude than the incident field, 
\mm{and the incident plane wave dominates in the total wave field $u=u^i+u^s$.
In such a case, it is straightforward to derive \eqref{eq:cc1}. However, we shall not explore more 
about this point. Finally, we also like to point out that our arguments for the uniqueness results in Theorems~\ref{thm:uniqueness1} and \ref{thm:uniqueness2} are ``localized" around the corner point $x_c$. Therefore one may consider other different types of wave incidences from the incident plane wave \eqref{eq:pw}, e.g., the point source of the form, 
\begin{equation}\label{eq:ps}
u^i(\mathbf x; \mathbf z_0)=H_0^1(k|\mathbf x-\mathbf z_0|),\  \mathbf x, \mathbf z_0\in\mathbb{R}^2,
\end{equation}
where $H_0^1$ is the zeroth-order Hankel function of the first kind, and 
$\mathbf z_0$ signifies the location of the source $u^i(\mathbf x, \mathbf z_0)$. 
$u^i(\mathbf x; \mathbf z_0)$ blows up at the point $\mathbf z_0$. By direct verifications, we can 
show that both the uniqueness results in Theorem~\ref{thm:uniqueness1} and \ref{thm:uniqueness2} 
still hold for this point source incidence.}

\subsection{Unique recovery for the inverse diffraction grating problem}

In this subsection, we consider the unique recovery for the inverse diffraction grating problem. First we give a brief review of the basic mathematical model for this inverse problem.   
Let the profile of a diffraction grating be described by the curve
$$
\Lambda_f=\{(x_1,x_2) \in \R^2 ;~x_2 =f(x_1)\},
$$
where $f$ is a periodic Lipschitz function with period $2\pi$. Let
$$
\Omega_f=\{ \mathbf  x\in \R^2; x_2 > f(x_1), x_1 \in \R\}
$$
be filled with a material whose index of refraction (or wave number) $k$ is a positive constant. Suppose further that the incident  wave given by 
\begin{equation}
	u^i(\mathbf x;k, \mathbf d)=e^{\bsi k \mathbf d \cdot \mathbf x},\quad \mathbf d=(\sin \theta, -\cos \theta)^\top ,\quad \theta \in \left( - \frac{\pi}{2} , \frac{\pi}{2} \right),
\end{equation}
propagates to $\Lambda_f$ from the top. Then the total wave satisfies the following Helmholtz system:
\begin{equation}\label{eq:76}
	\Delta u+k^2 u=0\hspace*{0,5cm} \mbox{in}\ \ \Omega_f; \quad  {\mathcal B}(u)\big|_{\Lambda _f}=0 \hspace*{0,5cm}  \mbox{on} \ \ \Lambda_f, 
	\end{equation}
with the generalized impedance boundary condition
\begin{equation}\label{eq:gbc}
\mathcal{B}(u)=\partial_\nu u+\eta u=0\quad\mbox{on}\ \ \partial\Omega,
\end{equation}
where $\eta$ can be $\infty$ or $0$, corresponding to a sound-soft or sound-hard grating, respectively.

To achieve uniqueness of \eqref{eq:76}, the total wave field $u$ should be $\alpha$-quasiperiodic in the  $x_1$-direction, with $\alpha=k\sin \theta$, which means that 
$$
u(x_1+2\pi, x_2)=e^{2\bsi \alpha \pi}\cdot u(x_1,x_2),
$$
and the scattered field $u^s$ satisfies  the Rayleigh expansion  (cf.\cite{millar69,millar73}):
\begin{align*}
		u^s(\mathbf x;k,\mathbf d )&=\sum_{n=-\infty}^{+\infty } u_n e^{\bsi { \xi}_n (\theta ) \cdot \mathbf  x } \quad \mbox{for} \quad x_2 > \max_{x_1\in [0, 2\pi]} f(x_1),
			\end{align*}
\mm{
where $u_n\in \C(n\in \mathbb Z)$ are called the Rayleigh coefficient of $u^s$, and} 
\begin{equation}\label{eq:839n}
\begin{split}
	\xi_n(\theta )&=\left(\alpha_n(\theta), \beta_n(\theta)\right)^\top, \quad \alpha_n(\theta )=n+k\sin \theta, \\
		 \beta_n(\theta )&=\left\{\begin{array}
			{c}
			\sqrt{k^2- \alpha_n^2 (\theta) }, \quad\mbox{ if } |\alpha_n (\theta )| \leq k\\
			\\[1pt]
			\bsi \sqrt{\alpha_n^2 (\theta)-k^2 }, \quad\mbox{ if } |\alpha_n (\theta)| > k
		\end{array}. \right. 
		\end{split}
\end{equation}

The existence and uniqueness of the $\alpha$-quasiperiodic solution to \eqref{eq:76} for the sound-soft or impedance  boundary condition with $\eta\in \C$ being a constant satisfying $\Im  (\eta) >0$ can be found in \cite{alber,cadilhac,kirsch93,kirsch94}.  It should be pointed out that the uniqueness of the direct scattering problem associated with the sound-hard condition is not always true (see \cite{kamotski}). In our subsequent study, we assume the well-posedness of the forward scattering problem and focus on the study of the inverse grating problem. 

Introduce a measurement  boundary as
 $$
 \Gamma_b:=\{ (x_1,b)\in \R^2 ;~ 0 \leq x_1 \leq 2 \pi, \, b> \max_{x_1\in [0, 2\pi]}|f(x_1)|  \}.
 $$
\mm{
The inverse diffraction grating problem is to determine $(\Lambda_f,\eta)$ from the knowledge of $u(\mathbf x|_{\Gamma_b};k,\mathbf d)$, and 
can be formulated as the operator equation:
\begin{equation}\label{eq:iopd}
\mathcal{F}(\Lambda_f,\eta )=u(\mathbf x;k,\mathbf d), \quad  \mathbf x\in \Gamma_b,
\end{equation}
where $\mathcal{F}$ is defined by the forward diffraction scattering system, and is nonlinear.
}

The unique recovery result on the inverse  diffraction  grating problem with the sound-soft boundary condition by  a finite number of incident plane waves  can be found in \cite{kirsch94,kirsch97}. 
\mm{
But the unique identifiability still open for the impedance or generalized impedance cases, and 
will be the focus of the remaining task in this work. 
To do so, 
we propose the following admissible polygonal gratings associated with 
the inverse diffraction grating problem.
}

\begin{definition}\label{def:dd1}
Let $(\Lambda_f, \eta)$ be a periodic grating as described above. If $f$ is a piecewise linear polynomial within one period, and on each piece of $\Lambda_f$, $\eta$ is either a constant (possibly zero) or $\infty$, 
then $(\Lambda_f, \eta)$ is said to be an admissible polygonal grating. 
\end{definition}

\begin{definition}\label{def:r1g}
Let $(\Lambda_f, \eta)$ be an admissible polygonal grating. Let $\Gamma^+$ and $\Gamma^-$ be two adjacent pieces of $\Lambda_f$. The intersecting point of  $\Gamma^+$ and $\Gamma^-$ is called a corner point of $\Lambda_f$, and $\angle(\Gamma^+,\Gamma^-)$ is called a corner angle. If all the corner angles of $\Lambda_f$ are irrational, then it is said to be an \emph{irrational polygonal grating}. If a corner angle 
of $\Lambda_f$ is rational, it is called a \emph{rational polygonal grating}. The smallest degree of the rational corner angles of $\Lambda_f$ is referred to as the \emph{rational degree} of $\Lambda_f$. 
\end{definition}

Clearly for a rational polygonal grating $\Lambda_f$ in Definition \ref{def:r1g}, the rational degree of $\Lambda_f$ is at least $2$. Next, we establish our uniqueness result in determining an admissible irrational polygonal grating by at most two incident waves. We first present a useful lemma, 
\mm{whose proof follows from a completely similar argument to that of \cite[Theorem 5.1]{CK}.}


\begin{lemma}\label{lem:83}
	Let $\mathbf \xi_{\ell}\in\mathbb{R}^2$, $\ell=1,\ldots, n$, be $n$ vectors 
	which are distinct from each other, $D$ be an open set in $\mathbb{R}^2$. 
	\mm{Then all the functions in the following set are linearly independent:}
	$$
	\{e^{\bsi \mathbf \xi_{\ell} \cdot \mathbf x} ;~\mathbf x \in D, \ \ \ell=1,2,\ldots, n \}
	$$
\end{lemma}

\begin{theorem}\label{thm:uniqueness1g}
Let $(\Lambda_f, \eta)$ and $(\Lambda_{\widetilde f}, \widetilde\eta)$ be two admissible irrational polygonal gratings, 
\mm{
$\mathbf{G}$ be the unbounded connected component of $\Omega_f\cap \Omega_{\widetilde f}$. 
}
Let $k\in\mathbb{R}_+$ be fixed and $\mathbf{d}_\ell$, $\ell=1, 2$ be two distinct incident directions from $\mathbb{S}^1$, with 
\begin{equation}\label{eq:8incident}
\mathbf{d}_\ell=(\sin \theta_\ell, -\cos \theta_\ell)^\top ,\quad \theta_\ell  \in \left( - \frac{\pi}{2} , \frac{\pi}{2} \right). 
\end{equation}
Let $u_b$ and $\widetilde{u}_b$ be the measurements on $\Gamma_b$ associated with $(\Lambda_f, \eta)$ and $(\widetilde\Lambda_{f} , \widetilde\eta)$, respectively. 
\mm{If it holds that}
\begin{equation}\label{eq:836}
	u(\mathbf  x;k,\mathbf d_\ell )=\widetilde u(\mathbf x;k,\mathbf  d_\ell ),\quad \ell=1,2,\quad  \mathbf x=(x_1, b) \in \Gamma_b, 
\end{equation}
\mm{
then it cannot be true that} there exists a corner point of $\Lambda_{f}$ lying on $\partial\mathbf{G}\backslash\partial\Lambda_{\widetilde f}$, or a corner point of $\Lambda_{\widetilde f}$ lying on $\partial\mathbf{G}\backslash \partial \Lambda_{f}$.

%
\end{theorem}

\begin{proof}
The proof follows from a similar argument to that for Theorem~\ref{thm:uniqueness1}, and 
we only sketch the necessary modifications in this new setup. 
By contradiction and without loss of generality, we assume that there exists a corner point $\mathbf{x}_c$ of $\Lambda_f$ which lies on $\partial\mathbf{G}\backslash\Lambda_{\widetilde f}$. 

First, by using the well-posedness of the forward problem and analytic continuation, we have from \eqref{eq:836} that $u(\mathbf  x;k,\mathbf d_\ell )=\widetilde u(\mathbf x;k,\mathbf  d_\ell )$ holds for $\mathbf{x}\in\mathbf{G}$. Using a similar argument to the proof of Theorem \ref{thm:uniqueness1}, we can prove that
	$$
	u(\mathbf x; k,\mathbf d_\ell) = 0 \quad \mbox{ or}\quad  v (\mathbf x) = 0 \quad \mbox{for} \quad x_2 > \max_{x_1\in [0, 2\pi]} f(x_1), 
	$$
	where $v$ is similarly defined to \eqref{eq:bb2} and \eqref{eq:bb1}. 
	Next, when $x_2  > \max_{x_1\in [0, 2\pi]}|f(x_1)| $, $u(\mathbf x;k, \mathbf d_\ell  )$ has the Rayleigh expansion (cf.\cite{millar69,millar73}): 
	\begin{align}\label{eq:839}
		u(\mathbf x;k,\mathbf d_\ell )&=e^{\bsi k \mathbf d_\ell  \cdot \mathbf x}+ \sum_{n=-\infty}^{+\infty } u_n e^{\bsi { \xi}_n (\theta_\ell  ) \cdot \mathbf  x }\quad \mbox{for} \quad x_2 > \max_{x_1\in [0, 2\pi]} f(x_1),
	\end{align}
	where $\xi_n(\theta_\ell  )$, $\alpha_n(\theta_\ell)$, $\beta_n(\theta_\ell )$ are defined in \eqref{eq:839n}. Using the definition of $\alpha_0(\theta_\ell)$ and  $\beta_0(\theta_\ell) $ in  \eqref{eq:839n}, we can easily show that 
\begin{equation}\label{eq:840}
	k\mathbf d_{\ell}=(\alpha_0(\theta_\ell), -\beta_0(\theta_\ell) )^\top . 
\end{equation}		

Next, we consider two separate cases. 

\medskip

\noindent {\bf Case 1.}	Suppose that either $u(\mathbf x_c; k, \mathbf{d}_1)$ or $u(\mathbf x_c; k, \mathbf{d}_2)$ is zero. Without loss of generality, we assume the former case. Then
$$
u(\mathbf x;k, \mathbf d_1)=0\quad \mbox{for} \quad x_2 > \max_{x_1\in [0, 2\pi]} f(x_1). 
$$
Clearly any two vectors of $\{\xi_n (\theta_1 )~|~n\in \mathbb Z\}$ are distinct from each other.  Moreover,  in view of \eqref{eq:840}, $k \mathbf d_1 \notin \{\xi_n(\theta_1 ) ~|~n\in \mathbb Z\}$	since $|\theta_1|<\pi/2$. In view of \eqref{eq:839}, from Lemma \ref{lem:83} we can arrive at a contradiction.

\medskip 
	
\noindent {\bf Case 2.}~ Suppose that both $u(\mathbf x_c; k, \mathbf{d}_1)\neq 0$ and $u(\mathbf x_c; k, \mathbf{d}_2)\neq 0$. Then it holds that
\begin{equation}\label{eq:841}
\alpha_1 u(\mathbf x; k, \mathbf{d}_1)+\alpha_2 u(\mathbf x; k, \mathbf{d}_2)=0  \quad \mbox{for} \quad  x_2 > \max_{x_1\in [0, 2\pi]} f(x_1), 
\end{equation}
where $\alpha_\ell  \neq 0$, $\ell=1,2$, are defined in \eqref{eq:bb1}. Substituting \eqref{eq:839} into \eqref{eq:841}, we derive that
\begin{align}\label{eq:842}
	\sum_{\ell=1}^2 \alpha_\ell  e^{\bsi k \mathbf d_\ell  \cdot \mathbf x}+ \sum_{n=-\infty}^{+\infty }  \sum_{\ell=1}^2u_n (\theta_\ell )\alpha_\ell  e^{\bsi { \xi}_n (\theta_\ell ) \cdot \mathbf  x }= 0 \quad \mbox{for} \quad  x_2 > \max_{x_1\in [0, 2\pi]} f(x_1), 
\end{align}
where $u_n(\theta_\ell)\in \C(n\in \mathbb Z)$ are the  Rayleigh coefficients of $u^s(\mathbf x; k, \mathbf d_\ell)$ associated with the incident wave $e^{\bsi k \mathbf d_\ell \cdot \mathbf x}$. 
Clearly, any two vectors of the set 
$$
\{ k \mathbf d_1\}\bigcup\{ k \mathbf d_2\} \bigcup   \{\xi_n (\theta_1 )~|~n\in \mathbb Z\} \bigcup \{\xi_n (\theta_2 )~|~n\in \mathbb Z\} 
$$
are distinct since $|\theta_\ell| < \pi/2$ and \eqref{eq:840}. 
\mm{
Using Lemma \ref{lem:83} and 
\eqref{eq:842}, we can see $\alpha_\ell=0$ for $\ell=1,2$, which is a contradiction to $\alpha_\ell \neq 0$, $\ell=1,2$.}  
%
\end{proof}

For the polygonal gratings, one can introduce a certain notion of ``convexity" in the sense that if two such gratings are different, then their difference must contain a corner point lying outside their union. Clearly, by Theorem~\ref{thm:uniqueness1g}, if a polygonal grating is ``convex", then both the grating and its surface impedance can be uniquely determined by at most two measurements.

%
%

\mm{
As the result in Theorem~\ref{thm:uniqueness2} for the inverse obstacle problem, we may consider 
the unique determination of an admissible rational polygonal grating by two measurement if a similar condition to \eqref{eq:cc1} is introduced in this new setup. In such a case, 
one can establish the local unique recovery result, similar to Theorem~\ref{thm:uniqueness1g}.
}

\section*{Acknowledgement}

The work  of H Diao was supported in part by the Fundamental Research Funds for the Central Universities under the grant 2412017FZ007.  The work of H Liu was supported by the FRG fund from Hong Kong Baptist University and the Hong Kong RGC General Research Fund (projects 12301218 and 12302017).  The work of J Zou 
was supported by the Hong Kong RGC General Research Fund (project 14304517) and 
NSFC/Hong Kong RGC Joint Research Scheme 2016/17 (project N\_CUHK437/16).

\end{document}